\documentclass[11pt, reqno]{amsart}
\usepackage{amssymb,amsthm,amsmath,amstext}
\numberwithin{equation}{section}
\usepackage{mathrsfs}  
\usepackage{bm}        
\usepackage{mathtools} 
\usepackage{tikz-cd}
\usepackage{comment}
\usepackage{amscd}
\usepackage[ruled,vlined]{algorithm2e}

\SetKwInOut{Require}{Input}
\SetKwInOut{Ensure}{Output}
\usetikzlibrary{arrows.meta, decorations.pathreplacing, calc, patterns}
\usepackage{xcolor}

\usepackage[left=1.28in,top=.8in,right=1.28in,bottom=.8in]{geometry}
\usepackage{float}
\usepackage{graphicx}
\usepackage[all]{xy}
\usepackage[title]{appendix} 

\usepackage[mathcal]{euscript}
\usepackage{txfonts}
\DeclareMathAlphabet{\matholdcal}{OMS}{cmsy}{m}{n}

\theoremstyle{plain}

\newtheorem{theorem}{Theorem}[section]
\newtheorem{prop}[theorem]{Proposition}
\newtheorem{lemma}[theorem]{Lemma}

\newtheorem{cor}[theorem]{Corollary}

\theoremstyle{definition}
\newtheorem{define}[theorem]{Definition}
\newtheorem{example}[theorem]{Example}

\theoremstyle{remark}
\newtheorem{remark}[theorem]{Remark}

\usepackage{tikz}
\usetikzlibrary{arrows.meta,positioning,fit,backgrounds,calc}

\newcommand{\Crit}{\operatorname{Crit}}
\newcommand{\searrowcollapse}{\searrow}

\usepackage{hyperref}
\usepackage{color}

\hypersetup{colorlinks=true,linkcolor=blue,anchorcolor=blue,citecolor=blue}

\begin{document}
	
	\title{Persistent Simple-homotopy invariants via discrete Morse theory}	
	
	\author{Divya Ahuja}
	\address{Institute of Mathematical Sciences \\ %
		C.I.T. Campus \\ %
		Taramani, Chennai 600113. India}
	\email{divyaa@imsc.res.in}
	
	\author{Priyavrat Deshpande}
	\address{Chennai Mathematical Institute, H1 SIPCOT IT Park, Siruseri, Tamil Nadu 603103, India} 
	\email{pdeshpande@cmi.ac.in}
\author{Jaya NN Iyer}
\address{The Institute of Mathematical Sciences,
	C.I.T. Campus, Taramani, Chennai 600113, India; and
	Max Planck Institute for Mathematics,
	Vivatsgasse 7, 53111 Bonn, Germany}
\email{jniyer@imsc.res.in, \, jniyer@mpim-bonn.mpg.de}
	\author{Siddharth Pritam}
	\address{Chennai Mathematical Institute, H1 SIPCOT IT Park, Siruseri, Tamil Nadu 603103, India} 
	\email{spritam@cmi.ac.in}

	\date{}

	\begin{abstract} 
		Persistent homology records the evolution of homological features along
		a filtration, but does not retain finer information related to
		simple-homotopy theory. In this paper, we develop two approaches to
		capturing such information for filtered simplicial complexes. We first
		introduce the Morse complexity profile, which records the minimal number
		of critical simplices at each filtration level. We study its invariance and stability properties and develop computable approximations using several discrete Morse matchings. We then introduce a persistent version of Whitehead torsion and show that it is
		invariant under both levelwise homotopy equivalence and interleaving equivalence
		of filtrations.

	\end{abstract} 
	
	\maketitle
	
	\medskip
	{MSC(2020) Subject Classification: 55N31, 57Q10, 57Q70, 55P10, 68U05}
	
	\medskip

	\section{Introduction}
	\noindent Persistent homology has emerged as one of the central tools of topological data analysis (TDA), providing a robust and computable summary of the multiscale homological structure of filtered spaces. 
	Given a data set, one first constructs a filtration of topological spaces that reflects the geometric information of the data. 
	Applying homology to this filtration yields a family of vector spaces together with structure maps induced by inclusions, forming a persistence module, namely a functor from a poset category to the category of vector spaces that records topological changes across scales. 
	Since its foundational development by Herbert Edelsbrunner et al.\ (see \cite{ELZ}), persistent homology has proven remarkably effective in applications ranging from geometry and dynamics to machine learning (see, for instance, \cite{ HA, EH, MN, OU, ZC05}). 
	Its success rests on two key properties: algebraic computability and stability under perturbation.
	However, persistent homology detects only additive invariants of filtered spaces, namely, Betti numbers and their persistence. As a consequence, it is insensitive to finer homotopy-theoretic structures. 
	In particular, persistent homology cannot distinguish between spaces that are homotopy equivalent but not simple-homotopy equivalent. This limitation is intrinsic: homology factors through abelianization and therefore discards the combinatorial complexity captured by simple-homotopy theory in the sense of Whitehead \cite{JHC}.
	
	\smallskip \noindent  Simple-homotopy theory, introduced by Whitehead, provides a combinatorial refinement of homotopy theory through elementary collapses and expansions of simplicial complexes. These elementary
	moves preserve homotopy type, while retaining information that is not detected by ordinary homotopy equivalence. 
	The obstruction to a homotopy equivalence being simple is measured by Whitehead torsion \cite{Coh,JHC}. The elementary collapses underlying simple-homotopy theory are also closely connected with discrete Morse theory.
	
	\smallskip
	\noindent Discrete Morse theory, introduced by Forman~\cite{RB98}, provides a combinatorial approach to the topology of simplicial complexes using
	real-valued functions on their simplices. The critical simplices of such
	a function determine a CW complex homotopy equivalent to the original
	complex, with one $p$-cell corresponding to each critical
	$p$-simplex. Moreover, when no critical simplices occur between two levels of a
	discrete Morse function, the change in the corresponding sublevel
	complexes is realized by a sequence of elementary collapses. The number of critical simplices therefore
	provides a natural measure of the complexity of such a reduction. These ideas suggest two related ways of looking beyond the information
	captured by persistent homology: through the evolution of discrete
	Morse complexity and through Whitehead torsion.

	\smallskip
	\noindent
	Discrete Morse theory has been used extensively in conjunction with
	persistent homology, particularly as a tool for reducing the size of
	filtered complexes while preserving persistence information.
	Mischaikow and Nanda~\cite{MN} extended combinatorial Morse theory
	from complexes to filtrations and developed an efficient 
	algorithm that reduces the number of cells in a filtered cell complex
	while preserving its persistent homology groups. Related Morse-theoretic reduction
	techniques have subsequently been developed for multidimensional
	persistence~\cite{AL17}, zigzag
	persistence~\cite{MS}, and computational settings involving
	large simplicial complexes and complex networks
	~\cite{KSRS, FIF}.
	
	\smallskip
	\noindent Our viewpoint is somewhat different: rather than using discrete Morse
	theory as a reduction tool for persistent homology, we study the
	evolution of Morse-theoretic complexity along the filtration itself.

	\smallskip
	\noindent
	Let $K$ be a finite simplicial complex. We begin by introducing the minimal Morse number of a simplicial complex $M(K)$ (see Definition \ref{D3.1}), defined as the minimum of the total number of critical simplices achieved by any discrete Morse function on the complex. Let
	$$
	K_\bullet: 
	K_0 \subseteq K_1 \subseteq \cdots \subseteq K_m=K
	$$
	be a filtration of $K$. We define the \emph{Morse complexity profile}
	$\matholdcal M(K_\bullet)$ (see Definition~\ref{def:mcp}) by recording the minimal
	Morse number at each filtration level. This profile captures combinatorial information beyond homology and allows us to detect filtration phenomena that remain invisible to persistent homology. In particular, it naturally highlights what we call \emph{Morse spikes} (Definition~\ref{def:morse spike}), namely filtration indices at which the homology remains unchanged while the minimal Morse number varies. The interaction of this profile with simple-homotopy theory is subtle:
	the minimal Morse number need not be preserved under simple-homotopy
	equivalence. Nevertheless, along a filtration, variations in the minimal Morse
	number can reveal changes in the underlying combinatorial structure
	that are not reflected in persistent homology. We demonstrate this
	through filtrations with identical persistent homology but distinct
	Morse complexity profiles.

	\smallskip
	\noindent
	This motivates our second construction which is based
	on Whitehead torsion, which is intrinsically associated with
	simple-homotopy equivalence. We restrict to portions of a filtration on which the inclusion maps
	$K_s\hookrightarrow K_t$ are homotopy equivalences and associate to
	each such inclusion its Whitehead torsion $\tau_{s,t}$ (see Definition \ref{def:PWT}). The
	classical composition law allows these torsion values to be viewed
	collectively, rather than as invariants of isolated maps, and
	together define the \emph{persistent Whitehead torsion} of the
	filtration. Conceptually, Morse theory and Whitehead torsion capture complementary features. Morse spikes detect the failure of collapsibility, while Whitehead torsion detects the failure of simple-homotopy equivalence. These developments show that tools from classical topology, particularly simple-homotopy
	theory, can significantly enrich the framework of TDA.
	
	\smallskip
	\noindent We now outline the structure of the paper in detail. In Section \ref{S2}, we begin with preliminaries on simplicial complexes, persistent homology, and discrete Morse theory in the sense of Forman. In Section \ref{S3},  we study the minimal Morse number, Morse complexity profile and their
	relation to collapsibility, simple-homotopy equivalence, and persistent
	homology. We first show that every collapsible simplicial complex admits minimal Morse data consisting of a single critical vertex and no higher-dimensional critical simplices, yielding minimal Morse number one (see Lemma \ref{lem:minmorse-collapse}). Furthermore, we prove that contractible but non-collapsible complexes necessarily possess a positive-dimensional critical simplex in every discrete Morse function, and hence have minimal Morse number at least two (see Lemma \ref{lem:minmorse-contract}). Motivated by these observations, one may ask whether the minimal Morse number is invariant under simple-homotopy equivalence. We answer this question negatively by showing that the Dunce hat (D) and a point are simple-homotopy equivalent while having different minimal Morse numbers (see Theorem \ref{T3.4}). We next
	study its behavior under collapses, proving that if $K\searrow L$, then
	$M(K)\leq M(L)$, and characterize when equality holds in terms of
	optimal discrete Morse matchings (where optimality means having the least number of critical simplices) (see Theorem \ref{thm:equality}). As a consequence, we obtain a
	sufficient matching-compatibility condition for the minimal Morse number
	to be preserved along sequences of elementary collapses and expansions.
	We conclude this section by showing that the Morse complexity profile carries information
	beyond persistent homology by constructing filtrations with identical
	persistent homology but distinct Morse complexity profiles (see Proposition~\ref{prop:morse-spike}).
	
	\smallskip\noindent In Section~\ref{S4}, we study the Morse complexity profile from both
	stability and computational perspectives. We first establish a point-wise stability result for the Morse complexity
	profile under adjacent transpositions of the filtration order (see Theorem \ref{thm:morse-transposition-stability}). Since computing the profile requires determining the minimal
	Morse number at every filtration level, which is computationally
	difficult in general, (see \cite{JP,PV}), we then turn to
	computable approximations of the profile.
	
	\smallskip
	\noindent
	We first consider apparent pairs \cite{Rip}, also known as shallow
	pairs \cite{EdelsDepth,EdelTranspositions}, and analyze the stability of the
	resulting profile (see Theorem \ref{thm:app}). We then introduce the last-facet profile, which
	provides a sharper upper bound on the minimal Morse number and a
	stronger point-wise stability bound (see Theorem \ref{thm:last-facet}). We also study the profile obtained
	by restricting a fixed, filtration-order-independent Morse matching.
	Finally, we consider a deterministic greedy acyclic matching procedure
	adapted to the filtration setting. Various greedy approaches to
	constructing discrete Morse matchings have been studied previously
	(see, for instance, \cite{LewinerThesis,LagoasEtAl}). We record the
	resulting number of critical simplices at each filtration level and
	show that, for a filtration containing $N$ simplices and having maximal
	dimension $d$, the algorithm has worst-case running time
	$O(d^2N^2)$ (see Theorem \ref{thm:greedy-algo}). Although it can provide a better approximation to the minimal Morse number, Proposition~\ref{prop:greedy-instability} shows that no uniform
	stability bound exists for the change in the number of critical simplices under
	a single adjacent transposition of the tie-breaking order.
	
	\smallskip
	\noindent
	Finally, in Section~\ref{S5}, we introduce a new invariant, called the \emph{persistent Whitehead torsion}, which refines the Morse-theoretic perspective developed earlier. The main issue in passing from classical Whitehead torsion to
	a filtration is that the homotopy equivalences occurring at different
	levels have varying fundamental groups, while their torsions must also
	be compatible with composition. We address this by introducing a
	\emph{torsion-eligible window} $J$, on which all inclusion maps are
	homotopy equivalences, and choosing compatible identifications of the
	fundamental groups with a fixed group $G$. The resulting torsion classes lie in a single Whitehead group
	$\mathrm{Wh}(G)$ and satisfy the composition law, allowing them to be
	organized into the persistent Whitehead torsion functor (see Definition \ref{def:PWT})
	\[
	\matholdcal{T}(K_\bullet):P(J)\longrightarrow
	\widehat{\mathrm{Wh}(G)}.
	\]
	where $\widehat{\mathrm{Wh}(G)}$ denotes the one-object category with endomorphism group $\mathrm{Wh}(G)$. This functor assigns every object of $P(J)$ to the unique object of
	$\widehat{\mathrm{Wh}(G)}$ and every morphism
	$s\le t$ to the Whitehead torsion
	$\tau_{s,t}$ of the inclusion
	$\iota_{s,t}:K_s\hookrightarrow K_t$. We show that its vanishing is equivalent to all inclusions in the
	window being simple-homotopy equivalences, and that the torsions of
	successive inclusions determine the entire functor (see Remark \ref{rem:elemtor}). Finally, we establish two invariance properties of the persistent Whitehead torsion functor. We show that it is preserved under level-wise simple-homotopy equivalences between filtrations (see Theorem \ref{thm:pwt-he} and Corollary \ref{cor:pwt-she}) and that, for two $\epsilon$-interleaved filtrations whose interleaving maps are homotopy equivalences, the corresponding persistent Whitehead torsion functors are naturally isomorphic (see Theorem \ref{thm:pwt-interleaving-he} and Corollary \ref{cor:pwt-interleaving-she}). We also present an example showing that persistent Whitehead torsion detects simple-homotopy information invisible to homology.

	\smallskip
	\noindent\textbf{Acknowledgments.}
	The authors thank Prof.~Parameswaran Sankaran for suggesting the lens
	space example considered in Section~\ref{S5}. The authors also thank
	Prof.~Nicholas Scoville for pointing out an error in an earlier proof in Section~\ref{S3}.
	
	\smallskip\noindent Jaya NN Iyer thanks Max-Planck Institute for Mathematics, Bonn, for their hospitality and support during the course of this work. 
	Priyavrat Deshpande and Siddharth Pritam are partially supported by a grant from the Infosys Foundation. Siddharth Pritam is also supported by ANRF ARG Matrix grant.
	
	\smallskip
	\noindent\textbf{Declaration on the use of generative AI.} Proof-checking of the article and the example (Proposition \ref{prop:greedy-instability}) used  ChatGPT (GPT-5.6
	Sol).
	
	\section{Preliminaries}\label{S2}
	\noindent
	Throughout this paper, $\mathbb{R}_{\geq 0}$
	denotes the set of non-negative real numbers, and $\mathbb{N}$ denotes
	the set of non-negative integers. In this section, we recall the basic notions from persistent homology,
	simple-homotopy theory, discrete Morse theory, and Whitehead torsion
	that will be used throughout the paper.
	
	\smallskip
	\noindent A simplicial complex consists of a set $K$ together with a collection $\mathcal{S}$ of subsets of $K$, called simplices, such that:
	
	\begin{itemize}
		\item[(i)] for every $v \in K$, the singleton $\{v\}$ belongs to $\mathcal{S}$; and
		
		\item[(ii)] whenever $\sigma \in \mathcal{S}$ and $\tau \subseteq \sigma$, then $\tau \in \mathcal{S}$.
	\end{itemize}
	\noindent
	The singletons $\{v\}$ are called the vertices of the complex.
	An element $\sigma \in \mathcal{S}$ is called a $k$-simplex (of dimension $k$) if $|\sigma| = k+1$ and we will sometimes denote it by $\sigma^k$.  
	If $\tau \subseteq \sigma$, then $\tau$ is called a face of $\sigma$, and $\sigma$ is called a coface of $\tau$.  If, in addition,
	$\dim\tau=\dim\sigma-1$, then $\tau$ is called a facet of $\sigma$,
	and $\sigma$ is called a cofacet of $\tau$. An orientation of a $k$-simplex $\sigma^k = \{v_0, \dots, v_k\}$ is defined as an equivalence class of orderings of its vertices, where two orderings $(v_0, \dots, v_k)$ and $(v_{\tau(0)}, \dots, v_{\tau(k)})$ are considered equivalent if the permutation $\tau$ has sign $+1$. An oriented simplex is denoted by $[\sigma]$. 
	A subcomplex of a simplicial complex $K$ is a subset $L \subseteq K$ that itself forms a simplicial complex.

	\smallskip
	\noindent A filtration of $K$ is a nested sequence of subcomplexes
	$$
	K_0 \subseteq K_1 \subseteq \cdots \subseteq K_m = K.
	$$
	For convenience, we assume $K_i = K_m$ for all $i \ge m$ in $\mathbb N$.
	A simplicial complex endowed with such a sequence is called a filtered simplicial complex. A map $f : K \to L$ between filtered simplicial complexes is called a filtration map if it preserves the filtration, i.e.,
	$
	f(K_i) \subseteq L_i$ for all $i.
	$
	In particular, $f$ restricts to maps
	$$
	f_i : K_i \to L_i
	$$
	for each $i$, and these are compatible with the inclusions.
	
	\smallskip
	\noindent
	\subsection{Persistent Homology.}
	
	\smallskip
	\noindent Given a filtered complex $K_0 \subseteq K_1 \subseteq \cdots \subseteq K_n = K,$ for each subcomplex $K_i$ of $K$ there exists a chain complex
	$$
	C_\bullet(K_i): \quad \cdots \longrightarrow C_r(K_i) \xrightarrow{\partial_{i_r}} C_{r-1}(K_i) \longrightarrow \cdots,
	$$
	where each $C_r(K_i)$ is the free abelian group on its set of oriented $r$-simplices. The map $\partial_{i_r}: C_r(K_i) \longrightarrow C_{r-1}(K_i)$ is the boundary operator defined linearly on chains as: for any $\sigma = [v_0, v_1, \dots, v_r]\in  C_r(K_i) $,
	$$\partial_{i_r} \sigma = \sum_{j=0}^{r} (-1)^j 
	[v_0, v_1, \dots, \widehat{v_j}, \dots, v_r],$$
	where $\widehat{v_j}$ denotes that the vertex $v_j$ is omitted from the simplex $\sigma$. Let $Z_{r}(K_i):= \ker(\partial_{i_r})$ be the cycle group and $B_{r}(K_i):=\operatorname{Im}(\partial_{i_{r+1}})$ be the boundary group. Then, the $r$-th homology group of $K_i$ is $$H_r(K_i):=Z_r(K_i)/B_r(K_i).$$
	On computing the homology groups of all the $K_i$, we obtain a sequence 
	$$
	\mathcal H_{r}(K): H_{r}(K_0)\longrightarrow H_{r}(K_1)\longrightarrow\cdots\longrightarrow H_{r}(K_n) ,
	$$
	connected by the morphisms $H_{r}(K_i)\longrightarrow H_{r}(K_{i+1})$ induced by the inclusion $K_i\hookrightarrow K_{i+1}$.
	The sequence of modules $\mathcal H_{r}(K)$ is called a persistence
	module. More formally, a persistence module over a ring $R$ consists of a sequence of $R$-modules
	$\{M_i\}_{i \ge 0}$ together with homomorphisms
	$ \varphi_i : M_i \longrightarrow M_{i+1}.$ 
	
	\smallskip
	\noindent For each $0\leq i\leq j\leq n$, the $r$-th persistent homology group is the subgroup of $H_r(K_j)$ given by
	\[
	H_r^{i,j}(K_\bullet)
	:=
	\operatorname{Im}\!\left(
	H_r(K_i)\longrightarrow H_r(K_j)
	\right).
	\]
	Thus, $H_r^{i,j}(K_\bullet)$ consists of the homology classes of $K_i$
	that are still alive in $K_j$. Equivalently,
	\[
	H_r^{i,j}(K_\bullet)
	\cong
	\frac{Z_r(K_i)}
	{B_r(K_j)\cap Z_r(K_i)}.
	\]
	
	
	\textit{Note:} Throughout this paper,  all homology groups are taken with coefficients
	in a field $\mathbb{F}$, and we write $H_r(K)$ instead of $H_r(K;\mathbb{F})$.

	\begin{remark}
		Every simplicial complex has a natural CW structure. Recall that a
		CW complex is built inductively by attaching $n$-cells, each homeomorphic
		to an open $n$-ball. More precisely, the $n$-skeleton is obtained from
		the $(n-1)$-skeleton by attaching closed $n$-balls $D^n$ along their
		boundaries $S^{n-1}=\partial D^n$. In the natural CW structure of a simplicial
		complex, the open $n$-simplices are precisely the $n$-cells.
		
	\end{remark}
	\smallskip
	\noindent
	\subsection{Homotopy equivalence and $\epsilon$-interleavings of filtered simplicial complexes} Recall that a continuous map $f:X \to Y$ between two topological spaces $X$ and $Y$ is a homotopy equivalence if there exists a map $g:Y \to X$ such that $g \circ f \simeq \mathrm{id}_X$ and $f \circ g \simeq \mathrm{id}_Y$, where $\simeq$ denotes the homotopy of maps. In this case, $X$ and $Y$ are said to be homotopy equivalent  or to
	have the same homotopy type, and we write $X \simeq Y$. 
	
	\smallskip
	\noindent
	Two simplicial complexes $K$ and $L$ are said to be homotopy equivalent,
	or to have the same homotopy type, if their geometric realizations
	$|K|$ and $|L|$ are homotopy equivalent. However, by abuse of notation, we continue to write $
	f:K\longrightarrow L$
	for the induced continuous map between geometric realizations
	$|K|$ and $|L|$.
	
	\smallskip\noindent
	Let $\emptyset=K_0 \subseteq K_1 \subseteq \cdots \subseteq K_n = K$ be a finite filtration of a simplicial complex $K$. Note that every finite $\mathbb{N}$-indexed filtration can be viewed as a finite
	$\mathbb{R}_{\ge0}$-indexed filtration by setting
	$K_t=K_{\lfloor t\rfloor}$, where $\lfloor t\rfloor$ denotes the greatest integer less than or equal to $t$, and defining the corresponding inclusion maps by
	\[
	\iota^K_{s,t}:=\iota^K_{\lfloor s\rfloor,\lfloor t\rfloor},
	\qquad s\le t.
	\]
	Conversely, every finite $\mathbb{R}_{\ge0}$-indexed filtration contains only finitely many distinct complexes, which may be enumerated to obtain an equivalent $\mathbb{N}$-indexed filtration. 
	
	\smallskip\noindent
	For the discussion of $\epsilon$-interleavings, we regard finite
	filtrations as $\mathbb{R}_{\geq 0}$-indexed filtrations. We recall the following definition of $\epsilon$-interleaving of filtered simplicial complexes from \cite{BS}. 
	\begin{define}[$\epsilon$- interleaving of filtered simplicial complexes]\label{def:interleaving}
		Let $\{K_i\}_{i\ge 0}$ and $\{L_i\}_{i\ge 0}$ be two $\mathbb{R}_{\geq 0}$-indexed filtrations of the simplicial complexes
		$K$ and $L$, respectively.
		They are said to be $\epsilon$-interleaved if, for $\epsilon \ge 0$ in $\mathbb R$, there exist two families of simplicial maps 
		$\{\phi_i : K_i \to L_{i+\epsilon} \mid i \ge 0\}$ and 
		$\{\psi_i : L_i \to K_{i+\epsilon} \mid i \ge 0\}$
		such that:
		
		\begin{itemize}
			\item[(i)] for all $i \le j$, 
			$\phi_j \circ \iota^K_{i,j} = \iota^L_{i+\epsilon,\, j+\epsilon} \circ \phi_i$
			and 
			$\psi_j \circ \iota^L_{i,j} = \iota^K_{i+\epsilon,\, j+\epsilon} \circ \psi_i$, and
			
			\item[(ii)]  for all $i \ge 0$, 
			$\psi_{i+\epsilon}\circ \phi_i =\iota^K_{i,\, i+2\epsilon}$ and 
			$\phi_{i+\epsilon}\circ \psi_i =\iota^L_{i,\, i+2\epsilon}$,
		\end{itemize}
		\noindent
		where $\iota^K_{i,j} : K_i \hookrightarrow K_j$ and $\iota^L_{i,j} : L_i \hookrightarrow L_j$
		denote the natural inclusion maps.
	\end{define}
	
	\subsection{Discrete Morse Theory for simplicial complex} In this subsection, we recall the classical discrete Morse theory for a finite simplicial complex as introduced by Forman in \cite{RB98, RB02}. From now onwards, we will always consider finite simplicial complexes.

	\begin{define}[Elementary collapse]
		Let $K$ and $L$ be finite simplicial complexes. An elementary collapse from $K$ to $L$ occurs if there exist simplices
		$\sigma^r,\tau^{r+1} \in K$ such that:
		
		\begin{itemize}
			\item[(i)] $\sigma$ is a proper face of $\tau$, and
			\item[(ii)] $\tau$ is the unique simplex of $K$ properly containing $\sigma$.
		\end{itemize}
		\noindent In this case, $L=K\setminus\{\sigma,\tau\}$. The simplex $\sigma$ is called a
		\emph{free face} of $\tau$, and $(\sigma,\tau)$ a \emph{free pair}.
	\end{define}
	\noindent A collapse, denoted by $K \searrow L$, is a sequence
	\[
	K = K_n \searrow_e K_{n-1} \searrow_e \cdots \searrow_e K_0 = L
	\]
	of elementary collapses from $K$ to $L$. The inverse of this operation is called an expansion and is denoted by $L \nearrow K$.
	A subcomplex $L$ obtained by a collapse is called a weak core of $K$.  A weak core is called minimal if it contains no free faces,
	and hence admits no further elementary collapse. A minimal
	weak core need not be unique (see \cite{XF}).

	\begin{define}[Simple-homotopy type]\label{D2.SHT}
		Let $K$ and $L$ be finite simplicial complexes. Then, $K$ and $L$ are said to have the same simple-homotopy type if there exists a finite sequence 
		$$
		K = K_n, K_{n-1}, \ldots, K_0 = L
		$$
		of
		elementary collapses and expansions from $K$ to $L$.
	\end{define}
	\noindent
	A homotopy equivalence $f:K\longrightarrow L$ between simplicial
	complexes is called a simple-homotopy equivalence if it is homotopic to
	a map induced by a finite sequence of elementary collapses and
	expansions.
	\begin{example}[Simple-homotopy type via elementary collapses and expansions]\label{Ex2.SHT}
		We illustrate Definition~\ref{D2.SHT} with a concrete sequence of elementary collapses and
		expansions. Let $K$ be the simplicial complex consisting of a triangle $\{a,b,c\}$ together
		with an attached edge $\{c,d\}$ (see Figure~\ref{fig:sht}). The complex $L= \{a,b\}$ (a single edge) is obtained from $K$
		by the sequence
		\[
		K \;=\; K_3 \;\searrow_e\; K_2 \;\searrow_e\; K_1 \;\searrow_e\; K_0 \;=\; L,
		\]
		of three elementary collapses:
		\begin{itemize}
			\item[(i)]   $K_3 \searrow_e K_2$: collapse the free face $\{d\}$ together with the
			unique coface $\{c,d\}$ (i.e.\ remove the free edge $\{c,d\}$ and vertex $\{d\}$);
			\item[(ii)]  $K_2 \searrow_e K_1$: collapse the free face $\{b,c\}$ together with the
			triangle $\{a,b,c\}$;
			\item[(iii)] $K_1 \searrow_e K_0$: collapse the free face $\{a,c\}$ (now a free edge),
			reducing the complex to the single edge $\{a,b\}$.
		\end{itemize}
		Since $L = \{a,b\}$ can be expanded back to $K$ by reversing these steps, $K$ and $L$ have
		the same simple-homotopy type. The sequence is illustrated in Figure~\ref{fig:sht}.
		
		\begin{figure}[H]
			\centering
			\begin{tikzpicture}[
				scale=0.95,
				every node/.style={font=\small},
				vtx/.style={circle, fill=black, inner sep=0pt, minimum size=4.5pt},
				lbl/.style={font=\scriptsize},
				arr/.style={-{Stealth[length=5pt,width=4pt]}, thick},
				colarr/.style={-{Stealth[length=5pt,width=4pt]}, thick, gray!60},
				]
				
				\begin{scope}[xshift=0cm]
					\fill[teal!12] (0,0) -- (1.2,0) -- (0.6,1.05) -- cycle;
					\draw[thick, gray!70] (0,0) -- (1.2,0) -- (0.6,1.05) -- cycle;
					\draw[thick, gray!70] (0.6,1.05) -- (1.05,1.75);
					\node[vtx] at (0,0) {};
					\node[vtx] at (1.2,0) {};
					\node[vtx] at (0.6,1.05) {};
					\node[vtx] at (1.05,1.75) {};
					\node[lbl, below left=1pt]  at (0,0)     {$a$};
					\node[lbl, below right=1pt] at (1.2,0)   {$b$};
					\node[lbl, left=2pt]        at (0.6,1.05) {$c$};
					\node[lbl, above=2pt]       at (1.05,1.75){$d$};
					\node[lbl] at (0.6,-0.55) {$K_3 = K$};
					\draw[orange!80!black, very thick] (0.6,1.05) -- (1.05,1.75);
				\end{scope}
				
				\draw[arr] (1.55,0.7) -- node[above, font=\scriptsize]{$\searrow_e$} (2.15,0.7);
				
				\begin{scope}[xshift=2.7cm]
					\fill[teal!12] (0,0) -- (1.2,0) -- (0.6,1.05) -- cycle;
					\draw[thick, gray!70] (0,0) -- (1.2,0) -- (0.6,1.05) -- cycle;
					\node[vtx] at (0,0) {};
					\node[vtx] at (1.2,0) {};
					\node[vtx] at (0.6,1.05) {};
					\node[lbl, below left=1pt]  at (0,0)     {$a$};
					\node[lbl, below right=1pt] at (1.2,0)   {$b$};
					\node[lbl, above=2pt]       at (0.6,1.05){$c$};
					\node[lbl] at (0.6,-0.55) {$K_2$};
					\draw[orange!80!black, very thick] (1.2,0) -- (0.6,1.05);
				\end{scope}
				
				\draw[arr] (4.25,0.7) -- node[above, font=\scriptsize]{$\searrow_e$} (4.85,0.7);
				
				\begin{scope}[xshift=5.4cm]
					\draw[thick, gray!70] (0,0) -- (1.2,0);
					\draw[thick, gray!70] (0,0) -- (0.6,1.05);
					\node[vtx] at (0,0) {};
					\node[vtx] at (1.2,0) {};
					\node[vtx] at (0.6,1.05) {};
					\node[lbl, below left=1pt]  at (0,0)     {$a$};
					\node[lbl, below right=1pt] at (1.2,0)   {$b$};
					\node[lbl, above=2pt]       at (0.6,1.05){$c$};
					\node[lbl] at (0.6,-0.55) {$K_1$};
					\draw[orange!80!black, very thick] (0,0) -- (0.6,1.05);
				\end{scope}
				
				\draw[arr] (7.0,0.7) -- node[above, font=\scriptsize]{$\searrow_e$} (7.6,0.7);
				
				\begin{scope}[xshift=8.1cm]
					\draw[thick, gray!70] (0,0) -- (1.2,0);
					\node[vtx] at (0,0) {};
					\node[vtx] at (1.2,0) {};
					\node[lbl, below left=1pt]  at (0,0)   {$a$};
					\node[lbl, below right=1pt] at (1.2,0) {$b$};
					\node[lbl] at (0.6,-0.55) {$K_0 = L$};
				\end{scope}
				
				\draw[decorate, decoration={brace, amplitude=5pt, mirror},
				gray!70, thick]
				(-0.2,-0.75) -- (9.5,-0.75)
				node[midway, below=6pt, font=\scriptsize, gray!80]
				{$K \searrow L$: same simple-homotopy type};
				
				\node[font=\scriptsize, align=left] at (10.4, 0.7)
				{\textcolor{orange!80!black}{\rule[0.4ex]{0.5cm}{1.5pt}} free face\\[2pt]
					\tikz\fill[teal!20](0,0)--(0.4,0)--(0.2,0.35)--cycle; filled $2$-simplex};
				
			\end{tikzpicture}
			
			\caption{A sequence of three elementary collapses showing that $K = K_3$ (a filled triangle with
				an attached edge) has the same simple-homotopy type as $L = K_0$ (a single edge $\{a,b\}$).
				At each step the orange edge is the free face being collapsed together with its unique coface.
				The filled teal region indicates a $2$-simplex (triangle).}
			\label{fig:sht}
		\end{figure}
	\end{example}

	\begin{define}[Discrete vector field] Let $K$ be a finite simplicial complex. A discrete vector field $V$ on $K$ is a collection of pairs $(\sigma^{(r)},\tau^{(r+1)})$ such that $\sigma$ is a face of $\tau$ and each simplex of $K$ belongs 
		to at most one such pair. Such a collection of pairs is also referred to as a matching on the simplices of $K$. 
		
		\smallskip\noindent
		Let $V$ be a discrete vector field on $K$. A $V$-path is a sequence of simplices
		$$\sigma^{r}_0, \tau^{r+1}_0,
		\sigma^{r}_1, \tau^{r+1}_1, \dots,
		\tau^{r+1}_p, \sigma^{r}_{p+1}$$
		such that for each $i = 0, \dots, p$,
		$$(\sigma^{r}_i, \tau^{r+1}_i) \in V
		\quad \text{and} \quad
		\tau^{r+1}_i \supset \sigma^{r}_{i+1} \neq \sigma^{r}_i.
		$$
		
	\end{define}
	\noindent
	Note that the last simplex does not need to belong 
	to any pair in $V$. If $p=0$, the path is said to be trivial. A $V$-path is said to be closed if $\sigma^{r}_k=\sigma^{r}_0$.
	
	\begin{remark}
		The process of computing a minimal weak core of a simplicial complex is quite restrictive as many complexes may have no free faces, and different collapse sequences may produce non-isomorphic weak cores. 
		Discrete Morse theory offers a combinatorial method for reducing simplicial complexes without altering their homotopy type. 
	\end{remark}
	\begin{define}[Discrete Morse function]
		A function $f : K \longrightarrow \mathbb{R}$ defined on the simplices of a simplicial complex $K$ is called a discrete Morse function if for every simplex $\sigma^r \in K$, the set
		$$
		\widetilde{\mathcal M}(\sigma)
		=
		\{\tau^{r-1} \subset \sigma^r : f(\tau) \ge f(\sigma)\}
		\cup
		\{\tau^{r+1} \supset \sigma^r : f(\tau) \le f(\sigma)\}
		$$
		contains at most one element.
	\end{define}
	\begin{define}[Critical simplex]
		A simplex $\sigma^r\in K$ is called critical if $\widetilde{\mathcal M}(\sigma)=\varnothing$.
	\end{define}
	\noindent Let $\mathcal K$ be the partially ordered set of simplices of $K$ ordered by the face relation. It is called the face poset of $K$. Its \emph{Hasse diagram} has
	the simplices of $K$ as vertices, with an edge $\sigma\to\tau$ whenever
	$\tau$ is a codimension-one face of $\sigma$. Every discrete Morse function determines a matching on $\mathcal K$ by pairing simplices $\sigma$ and $\tau$ whenever
	$\widetilde{\mathcal M}(\sigma)=\{\tau\}$. For a matching, the edges corresponding to matched pairs are reversed
	in the Hasse diagram. The matching is called \emph{acyclic} if the modified Hasse diagram contains no directed cycle. It is well known
	that discrete Morse functions determine acyclic matchings and,
	conversely, every acyclic matching is induced by a discrete Morse
	function (see \cite{CH, XF}).
	
	\smallskip
	\noindent A discrete Morse function can also be realized as a gradient vector field, in which non-critical simplices are paired with each other. 
	\begin{define}[Gradient vector field]
		Let $f$ be a discrete Morse function on a simplicial complex $K$.
		The gradient vector field $\widetilde{V}$ associated with $f$ is the
		collection of pairs $$\widetilde{V}:=\{(\sigma^{r}, \tau^{r+1})~|~\sigma^{r} \subset \tau^{r+1}~\textup{and}~ f(\sigma) \ge f(\tau)\}.$$
		Such a pair is called a gradient pair. 
	\end{define}
	\noindent
	Note that a simplex belongs to the gradient vector field
	if and only if it is non-critical.
	Naturally, every gradient vector field arising from a discrete Morse function is a discrete vector field. For converse, we have the following theorem.
	\begin{theorem}(see \cite[Theorem 3.5]{RB02}, \cite[Theorem 2.5.1]{SN19})\label{thm:dvf=gvf}
		A discrete vector field $V$ is the gradient vector field of a discrete Morse function if and only if 
		it contains no non-trivial closed $V$-paths.
	\end{theorem}
	\begin{example}[Gradient vector field]\label{Ex2.DVF}
		Let $K$ be the simplicial complex consisting of four vertices $\{v_0\}, \{v_1\}, \{v_2\}, \{v_3\}$, five edges
		$$
		e_1 = \{v_0, v_1\},\quad e_2 = \{v_1, v_2\},\quad e_3 = \{v_0, v_2\},\quad e_4 = \{v_2, v_3\},\quad e_5 = \{v_1, v_3\},
		$$
		and two triangles $\tau_1 = \{v_0, v_1, v_2\}$ and $\tau_2 = \{v_1, v_2, v_3\}$ (see Figure \ref{fig:dvf}).
		Define 
		$$
		V \;=\; \bigl\{(v_0,\, e_3),\; (v_2,\, e_2),\;  (e_1, \tau_1)\bigr\}.
		$$
		Each pair $(\sigma^r, \tau^{r+1})$ satisfies $\sigma \subset \tau$, and every simplex belongs to at most one pair, so $V$ is a valid discrete vector field. 
		The simplices $v_1, v_3, e_4$, $e_5$, and $\tau_2$ belong to no pair in $V$ and are therefore \emph{critical}. The field $V$ contains no closed $V$-paths, so by Theorem~\ref{thm:dvf=gvf}, it is the gradient vector field of a discrete Morse function. A non-trivial $V$-path is the sequence $e_1,\, \tau_1,\, e_2$: since $(e_1, \tau_1) \in V$ and $\tau_1 \supset e_2 \neq e_1$, the path condition is satisfied.
		
		\begin{figure}[H]\label{Fig2.12}
			\centering
			\begin{tikzpicture}[
				scale=1.15,
				every node/.style={font=\small},
				vtx/.style={circle, draw, fill=white, inner sep=2pt, minimum size=6pt},
				crit/.style={circle, draw=orange!80!black, dashed, thick,
					inner sep=4pt, minimum size=16pt},
				gpair/.style={-{Stealth[length=5pt]}, teal, thick},
				vpath/.style={-{Stealth[length=5pt]}, blue!70!black, thick, dashed},
				edge/.style={draw=gray!60, thick},
				]
				
				\coordinate (v0) at (0,0);
				\coordinate (v1) at (3,0);
				\coordinate (v2) at (1.5,2.2);
				\coordinate (v3) at (4,2.2);
				
				\fill[teal!10] (v0) -- (v1) -- (v2) -- cycle;
				\fill[orange!10] (v1) -- (v2) -- (v3) -- cycle;
				
				\draw[edge] (v0) -- (v1) -- (v2) -- cycle;
				\draw[edge] (v1) -- (v2) -- (v3) -- cycle;
				
				\node at (1.5, 0.75) {\textcolor{teal!70!black}{$\tau_1$}};
				\node at (3.15, 1.1) {\textcolor{orange!70!black}{$\tau_2$}};
				
				\coordinate (me1) at ($(v0)!0.5!(v1)$);
				\coordinate (me2) at ($(v1)!0.5!(v2)$);
				\coordinate (me3) at ($(v0)!0.5!(v2)$);
				\coordinate (me4) at ($(v2)!0.5!(v3)$);
				\coordinate (me5) at ($(v1)!0.5!(v3)$);
				
				\node[below=2pt] at (me1) {\scriptsize $e_1$};
				\node[right=2pt] at (me2) {\scriptsize $e_2$};
				\node[left=2pt]  at (me3) {\scriptsize $e_3$};
				\node[above=2pt] at (me4) {\scriptsize $e_4$};
				\node[right=2pt] at (me5) {\scriptsize $e_5$};
				
				\node[vtx] at (v0) {};
				\node[vtx] at (v1) {};
				\node[vtx] at (v2) {};
				\node[vtx] at (v3) {};
				
				\node[below left]  at (v0) {$v_0$};
				\node[below right] at (v1) {$v_1$};
				\node[above left]  at (v2) {$v_2$};
				\node[above right] at (v3) {$v_3$};
				
				\draw[gpair] (v0) -- ($(v0)!0.55!(v2)$)
				node[midway, left=1pt, teal!70!black, font=\scriptsize] {$(v_0,e_3)$};
				\draw[gpair] (v2) -- ($(v2)!0.55!(v1)$)
				node[midway, right=1pt, teal!70!black, font=\scriptsize] {$(v_2,e_2)$};
				\draw[gpair] (me1) -- (barycentric cs:v0=1,v1=1,v2=1)
				node[midway, below=18pt, teal!70!black, font=\scriptsize]
				{$(e_1,\tau_1)$};
				
				\node[crit] at (v1) {};
				\node[crit] at (v3) {};
				\node[crit] at (me4) {};
				\node[crit] at (me5) {};
				\node[crit, inner sep=7pt] at (barycentric cs:v1=1,v2=1,v3=1) {};
				
				\draw[vpath] (me1) .. controls (1.5,0.55) and (2.0,0.85) .. (me2)
				node[midway, above=2pt, blue!70!black, font=\scriptsize] {$V$-path};
				
			\end{tikzpicture}
			
			\caption{The discrete vector field $V \;=\; \bigl\{(v_0,\, e_3),\; (v_2,\, e_2),\;  (e_1, \tau_1)\bigr\}$ on $K$.
				Teal arrows are gradient pairs (vertex to incident edge).
				Dashed orange rings mark the five critical simplices $v_1, v_3, e_4,e_5,\tau_2$.
				The dashed blue arc is the $V$-path $e_1,\tau_1,e_2$.}
			\label{fig:dvf}
		\end{figure}
	\end{example}	
	
	\smallskip \noindent Let $f : K \to \mathbb{R}$ be a discrete Morse function. Then, $f$ induces a filtration of $K$ by sublevel subcomplexes given by: for any $c \in \mathbb{R}$, $K_c$ is the smallest subcomplex of $K$ containing all simplices $\sigma$ with $f(\sigma) \leq c$. The change in homotopy type across the filtration is characterized by the following result.
	
	\begin{lemma}[Discrete Morse Lemma]\cite[Theorem 3.3 and 3.4]{RB98}
		Let $f : K \longrightarrow \mathbb{R}$ be a discrete Morse function. Then, the following statements hold:
		
		\begin{itemize}
			\item[(i)] If the interval $f^{-1}((c,d])$ contains no critical simplices,
			then $K_d \searrow K_c$.
			\item[(ii)] If the interval $f^{-1}((c,d])$ contains exactly one critical
			$k$-simplex, then $K_d$ is obtained from $K_c$
			by attaching a $k$-cell.
		\end{itemize}
	\end{lemma}
	\noindent
	It follows that the non-critical intervals correspond to collapse phenomena,
	while critical simplices correspond to homotopy changes. Consequently, we have the following result. \begin{theorem}\cite[Theorem 2.5]{RB02}
		Let $K$ be a finite simplicial complex equipped with a discrete Morse function. Then, $K$ is homotopy equivalent to a CW complex having exactly one $p$-cell for each critical simplex of dimension $p$.
	\end{theorem}
	\begin{theorem}[Weak Morse Inequalities]\cite[Theorem 2.11]{RB02}\label{thm:morse-ineq}
		Let $K$ be a simplicial complex equipped with a discrete Morse function. 
		For each $p$, let $c_p$ denote the number of critical $p$-simplices. 
		Let $\mathbb{F}$ be a coefficient field and let $b_p:=\dim H_p(K;\mathbb F)$ denote the $p$-th Betti number of $K$. Then, the following statements hold:
		
		\begin{itemize}
			\item[(i)] For each $p=0,1,2,\dots,n$, where $n=\dim K$,
			$$
			c_p \ge b_p.
			$$
			\item[(ii)] The alternating sums of the critical simplices and Betti numbers agree:
			$$
			c_0 - c_1 + c_2 - \cdots + (-1)^n c_n
			=
			b_0 - b_1 + b_2 - \cdots + (-1)^n b_n
			= \chi(K),
			$$
			where $\chi(K)$ denotes the Euler characteristic of $K$.
		\end{itemize}
	\end{theorem}
	
	\subsection{Whitehead torsion and simple-homotopy equivalence.}\label{Sub2.3} 
	In this subsection, we recall the basic notions concerning the Whitehead
	group and Whitehead torsion that will be used later in Section \ref{S5}. 
	
	\smallskip
	\noindent
	Let $R$ be a ring, and for each $n \geq 1$, let $\textup{GL}_n(R)$ denote the group of invertible $n \times n$ matrices over $R$. There are natural inclusion maps $\operatorname{GL}_n(R)\hookrightarrow
	\operatorname{GL}_{n+1}(R)$. Define 
	$$\textup{GL}(R) := \operatorname{colim}_n \textup{GL}_n(R)$$
	as the colimit of these inclusion maps. Let $\mathrm{E}_n(R)\subset \mathrm{GL}_n(R)$ be the subgroup  generated by the elementary matrices
	$e_{ij}(r)$ where $i\neq j$ and $r\in R,$ and define $$\textup{E}(R) := \operatorname{colim}_n \textup{E}_n(R).$$ The group $\textup{E}(R)$ is the commutator subgroup of $\textup{GL}(R)$ (see \cite[\S 1]{JM}). The first algebraic $K$-group of $R$, denoted by $K_1(R)$, is defined as the abelianisation of $\textup{GL}(R)$ (see \cite{ALQ}) and is given by the quotient group
	$$
	K_1(R) := \textup{GL}(R) / \textup{E}(R).
	$$
	Let $G$ be any group and let $R$ be the group ring $\mathbb Z[G]$,  i.e., a set of all formal finite sums of the form $\sum_{i}a_ig_i$ where $a_i\in \mathbb Z$ and $g_i\in G$, with the product given by $$\left(\sum_i a_i g_i\right)\left(\sum_j b_j h_j\right) = \sum_{i,j} a_i b_j \, g_i h_j.$$ Then, by \cite[Definition 2.4]{ALQ}, the \textit{Whitehead group} of a group $G$, denoted by $\operatorname{Wh}(G)$ is the quotient group
	$$\operatorname{Wh}(G):=K_1(\mathbb Z[G])/\langle \pm g~|~g\in G\rangle.$$
	
	\smallskip
	\noindent
	We next recall Whitehead torsion. Let $K$ and $L$ be finite connected simplicial complexes, and let
	$f:K\to L$ be a homotopy equivalence. Choosing a basepoint identifies
	$\pi_1(K)\cong\pi_1(L)\cong G$. Let $\widetilde K$ and $\widetilde L$ denote
	the universal covers of $K$ and $L$, respectively. A lift
	$\widetilde f:\widetilde K\to\widetilde L$ induces a chain homotopy
	equivalence
	\[
	\widetilde{f}_*:C_*(\widetilde{K};\mathbb Z)
	\longrightarrow
	C_*(\widetilde{L};\mathbb Z)
	\]
	of finite free $\mathbb Z[G]$-chain complexes. 
	The \textit{Whitehead torsion} of $f$, denoted by $\tau(f)$, is the image of the chain homotopy equivalence $\widetilde{f}_{\ast}$ in $\operatorname{Wh}(G)$ under the composite
	\[
	\mathrm{GL}_n(\mathbb Z[G])
	\longrightarrow
	\mathrm K_1(\mathbb Z[G])
	\longrightarrow
	\overline{\mathrm K}_1(\mathbb Z[G])
	\longrightarrow
	\overline{\mathrm K}_1(\mathbb Z[G])/G = \operatorname{Wh}(G)
	\]
	where $\overline{\mathrm K}_1(\mathbb Z(G)) :=\mathrm K_1(\mathbb Z(G))/\langle \pm 1\rangle$ (see \cite[\S4.2 and Definition~4.3]{ALQ}). Since inner automorphisms act trivially on $\operatorname{Wh}(G)$, $\tau(f)$ is independent of the choice of basepoint. Moreover, the Whitehead torsion is invariant under homotopy of homotopy equivalences. More precisely, if
	$f,g:K\to L$
	are homotopic homotopy equivalences between finite connected simplicial complexes, then
	\[
	\tau(f)=\tau(g)
	\]
	(see \cite[Statement~22.1]{Coh}, \cite[Proposition~4.4]{ALQ}). 
	Furthermore, the Whitehead torsion provides a complete obstruction to simple-homotopy equivalence. More precisely, 
	\begin{theorem}\cite{Coh, JM, JHC,ALQ}\label{T2.18}
		A homotopy equivalence
		$f:K\longrightarrow L$ between finite connected simplicial complexes is a simple-homotopy equivalence
		if and only if
		$
		\tau(f)=0.
		$
	\end{theorem}
	\noindent For further details on the construction and properties of Whitehead torsion, we refer the reader to \cite{JM,Coh,ALQ}.

	\section{ Minimal Morse number and Morse Complexity Profile}\label{S3}
	\noindent In this section, we define the minimal Morse number and the Morse complexity profile, and establish their relationship with collapsibility and simple-homotopy equivalence.
	\begin{define} [Minimal Morse Number]\label{D3.1}
		Let $K$ be a finite simplicial complex. We define the minimal Morse number of $K$ by
		\begin{equation}M(K)=\min_{f}\left\{\sum_{i\ge 0} c_i(f)\;\middle|\; f \text{ is a discrete Morse function on } K \right\},
		\end{equation}
		where $c_i(f)$ denotes the number of critical $i$-simplices of the discrete Morse function $f$.
	\end{define}
	
	\smallskip
	\noindent 
	Recall that a topological space $X$ is contractible if it is homotopy equivalent to the one-point space. A simplicial complex $K$ is said to be collapsible if it admits a sequence of elementary collapses that reduces it to a single vertex. Furthermore, every collapsible complex is contractible, but the converse is not always true.
	
	\begin{lemma}\label{lem:minmorse-collapse}
		Let $K$ be a collapsible simplicial complex and $m_k(K)$ be the minimal number of critical $k$-simplices achievable. Then, $m_0(K)=1$ and $m_k(K)=0$ for $k>0$. Moreover, the minimal Morse number satisfies $M(K)=1$.
	\end{lemma}
	\begin{proof}
		Since $K$ is collapsible, there exists a sequence of elementary collapses
		$$K = K_n, K_{n-1} ,\ldots, K_0 = \{v\},$$
		$(\sigma_i,\tau_i)$, with
		$\sigma_i$ a free face of $\tau_i$ and
		$\dim \tau_i=\dim \sigma_i+1$. Let $
		V=\{(\sigma_i,\tau_i)\}_{i=1}^n
		$
		be the collection of all free face pairs arising in the collapse sequence. Since a free face belongs to
		exactly one coface, every simplex of $K$ appears in at most one pair in $V$,
		so $V$ is a well-defined vector field. Moreover, $V$ has no non-trivial closed $V$-paths. Indeed, suppose such a path exists and consider a pair $(\sigma,\tau)$ in the path that is removed last in the collapse sequence. The pair following it is removed earlier while $\tau$ is still present, so its lower simplex has two cofaces, contradicting its freeness. Now, by Theorem \ref{thm:dvf=gvf}, there exists a discrete Morse
		function $f$ on $K$ whose gradient vector field is $V$. Since every simplex except the final vertex $v$
		belongs to some pair $(\sigma_i,\tau_i)$ in $V$, $v$ is critical. Therefore,
		$m_0(K)\le 1$ and $m_k(K)=0$ for $k>0$.
		
		\smallskip
		\noindent
		Now, as a collapsible complex is connected, $\beta_0(K)=\dim H_0(K;\mathbb F)=1$. Let $c_k(f)$ denote the number of critical $k$-simplices of $f$. Then, by the weak Morse inequalities, Theorem \ref{thm:morse-ineq}, $c_0(f)\ge 1$ for every discrete Morse function $f$ on $K$. Taking the minimum over all such functions, we obtain $m_0(K) \ge 1$. Thus, $m_0(K)=1$ and $m_k(K)=0$ for $k>0$. Hence,
		$$
		M(K) = 1.
		$$
	\end{proof}
	\begin{lemma}\label{lem:minmorse-contract}
		Let $K$ be a contractible but non-collapsible simplicial complex. 
		Then, any discrete Morse function on $K$ has at least one critical simplex 
		of positive dimension. In particular, $M(K)\ge 2$.
	\end{lemma}
	\begin{proof}
		Suppose that there exists a discrete Morse function $f$ on $K$
		with no critical simplices of positive dimension. Then, all critical simplices of $f$
		are vertices. Since $K$ is contractible,  $\beta_0(K)=\dim H_0(K;\mathbb F)=1$ and $\beta_k(K)=\dim H_k(K;\mathbb F)=0$ for all $k>0$.  Then, by the weak Morse inequalities, Theorem \ref{thm:morse-ineq}, we obtain $c_0(f)=1.$
		This shows that $f$ must have only one
		critical simplex, a vertex. Therefore, every other simplex of $K$ belongs to a gradient pair in the
		discrete gradient vector field of $f$. Each such pair corresponds to an
		elementary collapse, and reduces $K$ to the critical vertex. Hence, $K$ is collapsible, which gives a contradiction. Thus, any discrete Morse function on $K$ must have at least one critical
		simplex of positive dimension, and therefore $M(K)\ge 2$.
	\end{proof}
	
	\begin{theorem}\label{T3.4}
		Let $K$ be a finite simplicial complex. Then, we have the following statements.
		\begin{itemize}
			\item[(i)] If $K$ is collapsible, then $M(K) = 1$.
			\item[(ii)] If $K$ is contractible but not collapsible, then $M(K) \ge 2$.
		\end{itemize}
		In particular, the minimal Morse number $M$ is \textbf{not} an invariant of simple-homotopy type.
	\end{theorem}
	\begin{proof}
		Part (i) is Lemma~\ref{lem:minmorse-collapse} and part (ii) is Lemma~\ref{lem:minmorse-contract}.
		To see that $M$ is not a simple-homotopy invariant, consider the triangulation of the Dunce hat $D$,  which is a finite $2$-dimensional simplicial complex that is contractible but not collapsible.
		Since $D$ is contractible, $\pi_1(D)=0$ and hence
		$\mathrm{Wh}(0)=0$. Therefore, $D$ is simple-homotopy equivalent to a point (see \cite{JM}).
		However, $D$ is not collapsible, so by part (ii), $M(D)\ge 2$, while $M(\ast)=1$.
		Hence, $M(D)\ne M(\ast)$ even though $D$ and $\ast$ are simple-homotopy equivalent.
	\end{proof}

	\smallskip
	\noindent
	In the following Lemma~\ref{obs:inequality} and Theorem ~\ref{thm:equality} we establish the precise combinatorial conditions under which the minimal Morse number is invariant. 
	
	\begin{lemma}\label{obs:inequality}
		If $K\searrowcollapse L$, then $
		M(K)\leq M(L).$
	\end{lemma}
	
	\begin{proof} We prove our claim by extending an optimal Morse matching on $L$ to $K$.
		Choose an elementary collapse sequence
		\[
		K=K_0\searrow_{e} K_1\searrow_{e}\cdots
		\searrow_{e}K_m=L,
		\]
		where the $i$-th step removes the free pair
		$(\alpha_i,\beta_i)$, with $\alpha_i\subseteq\beta_i$.  If $\mathcal M_L$ is an optimal acyclic matching on $L$, i.e., an
		acyclic matching having exactly $M(L)$ critical simplices, we define
		\[
		\mathcal M  =\mathcal M_L\cup\{(\alpha_i,\beta_i):1\leq i\leq m\}.
		\]
		The pairs are disjoint, so $\mathcal M$ is a matching.  We verify its
		acyclicity by recording when each simplex disappears. Define
		\[
		r(\sigma)=
		\begin{cases}
			i, & \sigma\in\{\alpha_i,\beta_i\},\\
			m+1, & \sigma\in L.
		\end{cases}
		\]
		If $\sigma\subseteq\tau$, then
		\begin{equation}\label{eq:monotone-removal}
			r(\sigma)\geq r(\tau).
		\end{equation}
		Indeed, if $\tau$ survived strictly longer than its face $\sigma$, then at
		some stage of the collapse sequence the remaining complex would contain
		$\tau$ but not the face $\sigma$, contradicting that every $K_i$ is a
		simplicial complex.
		
		\smallskip
		\noindent We orient the Hasse diagram associated with $\mathcal M$ such that the unmatched incidences
		are directed downward, $\tau\to\sigma$, and matched incidences are reversed,
		$\sigma\to\tau$.  Along an unmatched arrow, equation \eqref{eq:monotone-removal}
		shows that $r$ is non-decreasing and along a matched arrow, $r$ is constant. Thus, $r$ is non-decreasing along every directed path.
		It follows that
		every directed cycle must lie entirely within a single level set of
		$r$. Indeed, if
		\[
		\sigma_0\to\sigma_1\to\cdots\to\sigma_k=\sigma_0
		\]
		is a directed cycle, then
		\[
		r(\sigma_0)\leq r(\sigma_1)\leq\cdots\leq
		r(\sigma_k)=r(\sigma_0),
		\]
		and hence all these values are equal.
		
		\smallskip
		\noindent
		For each $1\leq i\leq m$, the level set $r^{-1}(i)$ consists only of
		the pair $\{\alpha_i,\beta_i\}$ and therefore does not contain a directed
		cycle. The last level $r^{-1}(m+1)=L$ carries precisely the acyclic matching $\mathcal M_L$.
		Hence, no level set contains a directed cycle, and consequently
		$\mathcal M$ is acyclic.

		\smallskip
		\noindent  Equivalently, a path may enter $L$, but it cannot leave again since entering $L$
		raises $r$ to its maximal value $m+1$, whereas every directed arrow preserves
		or increases $r$.  This rules out precisely the proposed mixed cycle that
		enters through one boundary face and exits through another.

		\smallskip
		\noindent Finally, every simplex of $K\setminus L$ is paired by a collapse pair, so the
		critical simplices of $\mathcal M$ are exactly those of $\mathcal M_L$.
		Hence, $K$ admits an acyclic matching with $M(L)$ critical simplices, and
		\[
		M(K)\leq M(L).
		\]
		
		\begin{figure}[H]\label{fig:inequality}
			\centering
			\begin{tikzpicture}[
				pair/.style={draw, rounded corners, minimum width=1.55cm,
					minimum height=0.62cm, inner sep=2pt},
				region/.style={draw, rounded corners, inner sep=8pt},
				matched/.style={-{Latex[length=2.2mm]}, very thick},
				down/.style={-{Latex[length=2.2mm]}, dashed, thick},
				note/.style={align=left, font=\small}
				]
				\node[pair] (a1) {$\alpha_1$};
				\node[pair, right=1.25cm of a1] (b1) {$\beta_1$};
				\draw[matched] (a1) -- node[above,font=\scriptsize]{matched} (b1);
				\node[left=0.55cm of a1] {$r=1$};
				
				\node[pair, below=1.05cm of a1] (a2) {$\alpha_2$};
				\node[pair, right=1.25cm of a2] (b2) {$\beta_2$};
				\draw[matched] (a2) -- (b2);
				\node[left=0.55cm of a2] {$r=2$};
				\draw[down] (b1.south) -- (a2.north);
				
				\node[below=0.65cm of a2] (dots) {$\vdots$};
				
				\node[pair, below=0.65cm of dots] (am) {$\alpha_m$};
				\node[pair, right=1.25cm of am] (bm) {$\beta_m$};
				\draw[matched] (am) -- (bm);
				\node[left=0.55cm of am] {$r=m$};
				\draw[down] (b2.south) -- (dots.north);
				\draw[down] (dots.south) -- (am.north);
				
				\node[region, below=1.1cm of $(am)!0.5!(bm)$,
				minimum width=4.7cm, minimum height=1.15cm] (L)
				{$L$ with the acyclic matching $\mathcal M_L$};
				\node[left=0.55cm of L] {$r=m+1$};
				\draw[down] (bm.south) -- (L.north);
				
				\begin{scope}[on background layer]
					\node[region, fit=(a1)(b1)(a2)(b2)(dots)(am)(bm),
					label={[font=\small]above:$K\setminus L$}] (outside) {};
				\end{scope}
				
				\node[note, right=1.1cm of b2] {
					matched arrows stay at one level;\\
					all arrows between different levels\\
					point toward larger removal time.
				};
			\end{tikzpicture}
			\caption{A directed path may move toward later removal times and may enter
				$L$, but it can never move back to an earlier level. Hence, a path that enters
				$L$ cannot leave it and close a mixed cycle.}
			\label{fig:removal-time}
		\end{figure}
		\noindent The Figure \ref{fig:inequality} above illustrates the construction used in this proof.
	\end{proof}
	
	\begin{remark}[What the argument does and does not prove]\label{rem:matching}
		Strictly speaking, we extend the \emph{Morse matching} (or gradient vector
		field), not a prescribed numerical Morse function. Acyclicity of the extended
		matching guarantees that it is realized by a discrete Morse function.
		The extension has exactly $M(L)$ critical simplices, but it need not be
		optimal on $K$; an optimal matching on $K$ may use pairs crossing between
		$L$ and $K\setminus L$.
	\end{remark}
	
	\begin{define}
		Let $L\subseteq K$ be a subcomplex. A discrete Morse matching $\mathcal M$ on
		$K$ is said to \emph{respect} $L$ if no matched pair has one simplex in $L$
		and the other in $K\setminus L$. Equivalently, every matched pair lies either
		entirely inside $L$ or entirely outside $L$.
	\end{define}
	
	\begin{theorem}\label{thm:equality}
		Let $K\searrowcollapse L$. Then,
		\[
		M(K)=M(L)
		\]
		if and only if there exists an optimal discrete Morse matching on $K$ that
		respects $L$.
	\end{theorem}
	
	\begin{proof}
		Suppose first that $M(K)=M(L)$. We extend an optimal matching on $L$ by the
		collapse pairs, as in Lemma~\ref{obs:inequality}. The resulting matching
		on $K$ has $M(L)=M(K)$ critical simplices, so it is optimal; by construction,
		it respects $L$.
		
		\smallskip\noindent Conversely, let $\mathcal M$ be an optimal matching on $K$ that respects $L$.
		Its restriction $\mathcal M|_L$ is an acyclic matching on $L$. Moreover, since
		there are no cross-boundary matched pairs, its critical simplices are exactly
		the critical simplices of $\mathcal M$ that lie in $L$. Therefore,
		\[
		M(L)\leq \#\Crit(\mathcal M|_L)
		\leq \#\Crit(\mathcal M)=M(K).
		\]
		Together with Lemma~\ref{obs:inequality}, this yields $M(K)=M(L)$.
	\end{proof}
	
	\begin{remark}
		The equality criterion isolates the only possible source of strict inequality.
		Although the collapse pairs together with an optimal matching on $L$ give an
		acyclic matching on $K$ with $M(L)$ critical simplices, a better matching on $K$ may pair simplices across
		the boundary between $L$ and $K\setminus L$. Thus, $M(K)<M(L)$ can occur only
		when every optimal matching on $K$ uses such cross-boundary behaviour.
	\end{remark}
	\begin{remark}
		Since an elementary expansion is the inverse of an elementary collapse,
		Theorem~\ref{thm:equality} applies equally to an expansion. More precisely, if
		$K\nearrow L$, then 
		$$
		M(L)=M(K)
		$$
		if and only if there exists an optimal discrete Morse matching on $L$
		that respects $K$.
	\end{remark}
	\begin{cor}Let
		$$
		K=K_0,K_1,\ldots,K_m=L
		$$
		be a sequence of elementary collapses and expansions. Suppose that, at
		each step, the larger complex admits an optimal discrete Morse matching
		that respects the smaller complex. Then
		$$
		M(K)=M(L).
		$$
		In particular, the minimal Morse number is preserved under a
		simple-homotopy equivalence satisfying this matching compatibility
		condition.
	\end{cor}
	\begin{proof}
		For each $i$, either $K_i\searrow K_{i+1}$ or
		$K_i\nearrow K_{i+1}$. In both cases, Theorem~\ref{thm:equality} and the
		matching hypothesis give
		$
		M(K_i)=M(K_{i+1}).
		$ Applying this successively along the sequence yields
		$
		M(K)=M(L).$
	\end{proof}
	
	\smallskip
	\noindent
	We next define the Morse complexity profile of a filtration.
	\begin{define}[Morse Complexity profile]\label{def:mcp}
		Let $K$ be a finite simplicial complex and $K_\bullet:
		K_0 \subseteq K_1 \subseteq \cdots \subseteq K_n = K
		$ be a filtration of $K$. We define the \emph{Morse complexity profile} by
		\begin{equation}
			\matholdcal M(K_\bullet)=\{M(K_t)\}_{t=0}^{n}.
		\end{equation}
	\end{define}
	
	\begin{example}[Computing the Morse complexity profile]\label{Ex3.MCP}
		Consider the filtration
		\[
		K_\bullet: K_0\subset K_1\subset K_2\subset K_3
		\]
		shown in Figure~\ref{fig:mcp}, where $K_0=\{v_0\}$, $K_1$ is the edge
		$\{v_0,v_1\}$ together with its vertices, $K_2$ is the boundary of the
		triangle on $\{v_0,v_1,v_2\}$, and $K_3$ is the filled triangle. Here, we compute the Morse complexity profile as follows.
		
		\smallskip 
		\noindent
		Clearly, $M(K_0)=1$. Since $K_1$ and $K_3$ are collapsible,
		Lemma~\ref{lem:minmorse-collapse} gives
		\[
		M(K_1)=M(K_3)=1.
		\]
		On the other hand, $K_2\cong S^1$, so $\beta_0(K_2)=\beta_1(K_2)=1$.
		By the weak Morse inequalities, we get $M(K_2)\geq 2$. This bound is attained
		by the gradient vector field
		$
		V_2 = \bigl\{(\{v_1\}, \{v_0,v_1\}),\; (\{v_2\}, \{v_0,v_2\})\bigr\},
		$
		which leaves $\{v_0\}$ and $\{v_1,v_2\}$ as the only critical simplices. Hence, $M(K_2) = 2$.
		Therefore, the Morse complexity profile is 
		$$
		\matholdcal M(K_\bullet) \;=\; \bigl(M(K_0),\, M(K_1),\, M(K_2),\, M(K_3)\bigr) \;=\; (1,\, 1,\, 2,\, 1).
		$$
		This example illustrates the direct relationship between topological features (the $1$-cycle in $K_2$) and the jump in $M(K_2)$.

		\begin{figure}[H]
			\centering
			\begin{tikzpicture}[
				scale=1.65,
				vtx/.style={circle, draw=gray!80, fill=white, inner sep=0pt, minimum size=5pt},
				edge/.style={draw=gray!70, thick},
				crit/.style={circle, draw=orange!80!black, dashed, thick, inner sep=3pt, minimum size=13pt},
				gpair/.style={-{Stealth[length=4.5pt]}, teal!80!black, thick},
				lbl/.style={font=\scriptsize},
				]
				
				\begin{scope}[xshift=0cm]
					\node[vtx] (A0) at (0,0) {};
					\node[crit] at (0,0) {};
					\node[lbl, below=6pt] at (0,0) {$v_0$};
					\node[lbl, below=22pt, align=center] at (0,0) {$K_0$\\$M=1$};
				\end{scope}
				
				\draw[->, thick, gray!60] (0.45,0) -- node[above, lbl]{} (0.9,0);
				
				\begin{scope}[xshift=1.4cm]
					\draw[edge] (0,0) -- (1.2,0);
					\node[vtx] at (0,0) {};
					\node[vtx] at (1.2,0) {};
					\node[crit] at (0,0) {};
					\draw[gpair] (1.2,0) -- (0.65,0)
					node[midway, above=1pt, teal!70!black, lbl]{$(\{v_1\},e_1)$};
					\node[lbl, below=6pt] at (0,0)   {$v_0$};
					\node[lbl, below=6pt] at (1.2,0) {$v_1$};
					\node[lbl, below=22pt, align=center] at (0.6,0) {$K_1$\\$M=1$};
				\end{scope}
				
				\draw[->, thick, gray!60] (3.25,0) -- node[above, lbl]{} (3.7,0);
				
				\begin{scope}[xshift=4.2cm]
					\coordinate (p0) at (0.5,1.1);
					\coordinate (p1) at (0,0);
					\coordinate (p2) at (1.1,0);
					
					\draw[edge] (p0) -- (p1) -- (p2) -- (p0);
					
					\node[vtx] at (p0) {};
					\node[vtx] at (p1) {};
					\node[vtx] at (p2) {};
					
					\node[crit] at (p1) {};
					\node[crit] at ($(p0)!0.5!(p2)$) {};
					
					\draw[gpair] (p2) -- ($(p2)!0.55!(p1)$)
					node[midway, below=5pt, teal!70!black, lbl]{$(\{v_2\},\{v_0,v_2\})$};
					\draw[gpair] (p0) -- ($(p0)!0.55!(p1)$)
					node[midway, left=1pt, teal!70!black, lbl]{$(\{v_1\},\{v_0,v_1\})$};
					
					\node[lbl, above=2pt] at (p0)  {$v_1$};
					\node[lbl, below left=1pt] at (p1)  {$v_0$};
					\node[lbl, below right=1pt] at (p2) {$v_2$};
					\node[lbl, below=30pt, align=center] at ($(p1)!0.5!(p2)$) {$K_2$\\$M=2$};
				\end{scope}
				
				\draw[->, thick, gray!60] (5.75,0.4) -- node[above, lbl]{} (6.5,0.4);
				
				\begin{scope}[xshift=6.7cm]
					\coordinate (q0) at (0.5,1.1);
					\coordinate (q1) at (0,0);
					\coordinate (q2) at (1.1,0);
					
					\fill[teal!10] (q0) -- (q1) -- (q2) -- cycle;
					\draw[edge] (q0) -- (q1) -- (q2) -- (q0);
					
					\node[vtx] at (q0) {};
					\node[vtx] at (q1) {};
					\node[vtx] at (q2) {};
					
					\node[crit] at (q1) {};
					
					\draw[gpair] (q2) -- ($(q2)!0.55!(q1)$)
					node[midway, below=5pt, teal!70!black, lbl]{$(\{v_2\},\{v_0,v_2\})$};
					\draw[gpair] (q0) -- ($(q0)!0.55!(q1)$)
					node[midway, left=1pt, teal!70!black, lbl]{$$};
					\draw[gpair] ($(q0)!0.5!(q2)$) -- (barycentric cs:q0=1,q1=1,q2=1)
					node[midway, right=10pt, teal!70!black, lbl]{$(\{v_1,v_2\},\{v_0,v_1,v_2\})$};
					
					\node[lbl, above=2pt] at (q0) {$v_1$};
					\node[lbl, below left=1pt]  at (q1) {$v_0$};
					\node[lbl, below right=1pt] at (q2) {$v_2$};
					\node[lbl, below=30pt, align=center] at ($(q1)!0.5!(q2)$) {$K_3$\\$M=1$};
				\end{scope}
				
			\end{tikzpicture}
			\caption{Top: the filtration $K_0 \subset K_1 \subset K_2 \subset K_3$.
				Teal arrows show gradient pairs; dashed orange rings mark critical simplices.
				Bottom: the Morse complexity profile $\matholdcal M(K_\bullet) = (1,1,2,1)$.
				The peak at $t=2$ reflects the non-trivial $1$-cycle in $K_2 \cong S^1$.}
			\label{fig:mcp}
		\end{figure}
	\end{example}
	\begin{remark}
		We know from Theorem \ref{T3.4} that the minimal Morse number is not invariant under
		simple-homotopy equivalence. Consequently, The Morse complexity profile is therefore \textbf{\emph{not}} an
		invariant of level-wise simple-homotopy equivalence of filtrations. Indeed, two filtrations
		may be level-wise simple-homotopy equivalent while having different
		minimal Morse numbers at corresponding levels. Nevertheless, it remains a well-defined and computable invariant that detects collapsibility failures. 
	\end{remark}
	
	\begin{define}[Morse spike]\label{def:morse spike}
		Let $K_\bullet:
		K_0 \subseteq K_1 \subseteq \cdots \subseteq K_n = K
		$ be a filtration of a finite simplicial complex $K$. We say that a \emph{Morse spike} occurs at $t$ if:
		
		\begin{itemize}
			
			\item[(i)] $\,H_{\ast}(K_{t-1}) \cong H_{\ast}(K_t) \cong H_{\ast}(K_{t+1}),$
			
			\item[(ii)] $ M(K_t)>M(K_{t-1}),$ and
			
			\item[(iii)] $M(K_t) > M(K_{t+1})$.
		\end{itemize}
	\end{define}

	\begin{define}\cite[Definition 1.19]{VN}
		Let $K$ be a simplicial complex with vertex set $K_0$. 
		The cone over $K$, denoted $\mathrm{Cone}(K)$, is the simplicial complex 
		defined on the vertex set $K_0 \cup \{v\}$, where $v$ is a new vertex not 
		contained in $K_0$. For any $d>0$, 
		a $d$-simplex of $\mathrm{Cone}(K)$ is either a $d$-simplex of $K$ or a simplex $v\ast \sigma$
		obtained by adjoining $v$ to a $(d-1)$-simplex $\sigma$ of $K$.
	\end{define}
	\noindent
	It is well known that  the cone over any simplicial complex is contractible (see \cite[\S 2, Example 2.4]{VN}) and is also collapsible (see \cite[Lemma 2.3.4 \& Lemma 2.3.6]{AB}).
	
	\smallskip
	\noindent We now show that the persistent homology does not determine the Morse complexity profile. 
	\begin{prop}\label{prop:morse-spike}
		There exists a finite filtration
		\[
		K_\bullet=\{K_0\subset K_1\subset K_2\}
		\]
		of finite simplicial complexes such that all inclusion maps
		induce isomorphisms on homology, while the Morse complexity profile
		$\matholdcal M(K_\bullet)$ is non-constant.
	\end{prop}
	\begin{proof}
		We construct three simplicial complexes $K_0\subset K_1 \subset K_2$. First, let $K_0$ be a single vertex.
		Next, let $K_1=D$, where $D$ is the triangulation of Dunce hat.
		Identifying the unique vertex of $K_0$ with a vertex of $D$ gives an inclusion $K_0 \subset K_1$. Finally, let $K_2$ be the cone over $D$, i.e., $K_2:=\mathrm{Cone}(D)=v\ast D$, where $v$ is a new vertex not contained in the vertex set of $D$. In this complex every simplex
		of $D$ is joined with $v$ to form a higher-dimensional simplex. Since the cone over any simplicial complex is both contractible and collapsible,
		the complex $K_2$ is both contractible and collapsible. As $K_0$, $K_1$, and $K_2$ are all contractible, for each $i$, the homology groups satisfy $H_0(K_i)\cong \mathbb{F}$ and $H_k(K_i)=0$ for all $k>0$. Hence, the inclusion maps $K_0\hookrightarrow K_1$ and $K_1\hookrightarrow K_2$ induce isomorphisms on homology and the persistent
		homology of the filtration is trivial.
		
		\noindent We now compare their minimal Morse numbers. From Lemma \ref{lem:minmorse-collapse}, we get $M(K_0)=1$ and $M(K_2)=1$. Moreover, by Lemma \ref{lem:minmorse-contract}, we obtain $M(K_1)\ge 2$.  Therefore, the Morse complexity profile $(1,\ge 2,1)$ is not constant. This shows that a Morse spike occurs at $t=1$.
	\end{proof}
	\begin{lemma}\label{lem:morse-spike}
		Let $K_\bullet:
		K_0 \subseteq K_1 \subseteq \cdots \subseteq K_m $ be a filtration of finite simplicial complexes. Suppose that there exists an index $t$ such that $K_t$ is contractible but not collapsible, while $K_{t-1}$ and
		$K_{t+1}$ are collapsible. Then
		\[
		M(K_{t-1})=M(K_{t+1})=1
		\qquad\textup{and}\qquad
		M(K_t)\geq 2.
		\]
		Moreover, the inclusion maps
		$
		K_{t-1}\hookrightarrow K_t\hookrightarrow K_{t+1}
		$
		induce isomorphisms on the respective homology groups. Consequently, the Morse complexity
		profile has a Morse spike at $t$ which is not detected by persistent
		homology.
	\end{lemma}
	\begin{proof}
		Since $K_t$ is contractible but not collapsible, by Lemma \ref{lem:minmorse-contract}, $M(K_t)\ge 2.$ Furthermore, as $K_{t-1}$ and $K_{t+1}$ are collapsible, we have $M(K_{t-1})= M(K_{t+1})=1$ by Lemma \ref{lem:minmorse-collapse}. Moreover, since each of $K_{t-1}, K_t$, and $K_{t+1}$ is contractible, their homology groups satisfy $$
		H_0(K_i) \cong \mathbb{F}, \qquad H_k(K_i) = 0 \quad \text{for all } k \ge 1,
		$$
		for $i = t-1, t, t+1$. This shows that the Morse spike occurs at $t.$
	\end{proof}
	\noindent It follows from Lemma \ref{lem:morse-spike} that the persistent homology fails to detect non-collapsibility, whereas Morse complexity captures it.

	\section{Algorithmic framework and Stability results}\label{S4} 
	\noindent 
	The Morse complexity profile introduced in Section~\ref{S3} is defined
	in terms of the minimum number of critical simplices over all discrete
	Morse functions at each filtration level. This amounts to finding an optimal Morse matching, that is, an acyclic
	matching of the Hasse diagram with as few unmatched simplices as
	possible. The corresponding optimization problem is NP-hard, in fact,
	the associated decision problem is strongly NP-complete even for
	connected pure $2$-dimensional simplicial complexes embeddable in
	$\mathbb{R}^3$ \cite{JP}. Consequently, several optimization,
	heuristic, approximation, and parameterized approaches to optimal Morse
	matching have been studied (see, for instance,
	\cite{LLT,Hersh,JP,PV}).
	
	\smallskip
	\noindent
	Before turning to computable approximations and their stability, we first study the stability of the Morse complexity profile under changes in the simplex order. For this purpose, we work
	with simplex-wise filtrations, in which exactly one simplex is added
	at each step. More generally, any finite filtration equipped with a
	compatible total order on its simplices admits a simplex-wise
	refinement obtained by inserting the simplices one at a time in that
	order. We show that a single admissible adjacent transposition changes
	the minimal Morse number point-wise by at most $2$.
	
	\begin{theorem}\label{thm:morse-transposition-stability}
		Let $K_\bullet$ and $L_\bullet$ be two simplex-wise filtrations of the
		same finite simplicial complex whose filtration orders differ by a
		single adjacent transposition of two simplices $\sigma$ and
		$\tau$, occurring at positions $i$ and $i+1$. Then,
		\[
		\sup_t |M(K_t)-M(L_t)|\leq 2,
		\]
		where $M(K_t)$ denotes the minimal Morse number at each filtration level
		$t$.
	\end{theorem}
	
	\begin{proof}
		The two filtrations agree at every level except $i$. Write
		\[
		K_{i-1}=L_{i-1}=:X,\qquad
		K_i=X\cup\{\sigma\},\qquad
		L_i=X\cup\{\tau\},
		\]
		so that
		\[
		K_{i+1}=L_{i+1}=X\cup\{\sigma,\tau\}.
		\]
		We first observe that adjoining a single simplex changes the minimal
		Morse number by at most one. Indeed, an optimal Morse matching on $X$ can always be extended to
		$X\cup\{\sigma\}$ by leaving $\sigma$ critical. 
		Conversely, restricting an optimal matching on
		$X\cup\{\sigma\}$ to $X$ either removes the critical simplex $\sigma$,
		or removes the matched pair containing $\sigma$, in which case its
		partner becomes critical. Hence,
		\[
		|M(X\cup\{\sigma\})-M(X)|\leq 1.
		\]
		Since the two complexes differ by one simplex, their minimal Morse
		numbers have opposite parity, so the change is either $+1$ or $-1$.
		
		\smallskip
		\noindent
		Since $K_i$ and $L_i$ are each obtained from $X$ by adjoining a single
		simplex, the preceding observation gives
		\[
		|M(K_i)-M(X)|\leq 1,
		\qquad\textup{and}\qquad
		|M(L_i)-M(X)|\leq 1.
		\]
		Hence, by the triangle inequality,
		\[
		|M(K_i)-M(L_i)|
		\leq |M(K_i)-M(X)|+|M(X)-M(L_i)|
		\leq 2.
		\]
		Since $K_t=L_t$ for $t\neq i$, the result follows.
		
	\end{proof}
	\begin{remark}
		The same-dimension assumption guarantees that the adjacent
		transposition is admissible. More generally, the result holds for any
		adjacent transposition for which the transposed order still defines a
		simplex-wise filtration.
	\end{remark}
	
	\smallskip
	\noindent
	The preceding theorem shows that the Morse complexity profile is
	pointwise stable under an adjacent transposition. However, computing
	$M(K_t)$ requires finding an optimal Morse matching at every filtration
	level, which is computationally intractable in general. This motivates
	the study of computable approximations and their stability properties.
	We consider four approaches: \emph{apparent pairs}, the
	\emph{last-facet matching}, restriction of a \emph{fixed global
		matching}, and a \emph{greedy acyclic matching}. The first three admit
	uniform pointwise stability bounds, whereas the greedy construction is generally not
	stable under an adjacent transposition of its tie-breaking order.

	\smallskip
	\noindent
	\textbf{1. Apparent-pair approximation.} We first consider an approximation based on apparent pairs \cite{Rip}. Apparent pairs are a particularly simple class of persistence pairs
	that can be identified directly from the filtration order, without
	performing matrix reduction. Bauer \cite[\S 3.5]{Rip} used apparent
	pairs to accelerate the computation of Vietoris-Rips persistence. More recently, Edelsbrunner et al.\
	\cite{EdelsDepth,EdelTranspositions} studied these pairs, under the name \emph{shallow pairs}, in connection with cancellations, depth posets, and transpositions of filtration orders.
	
	\smallskip \noindent
	Since apparent pairs depend on the total order of the simplices, we continue to work with simplex-wise filtrations in this subsection. Let
	\[
	K_\bullet:K_0\subset K_1\subset\cdots\subset K_N=K
	\]
	be a simplex-wise filtration of a finite simplicial complex $K$. 
	We denote the
	resulting filtration order of the simplices by
	\[
	\sigma_1<\sigma_2<\cdots<\sigma_N.
	\]
	\begin{define}[Apparent pair {\cite[Definition 3.2]{Rip}}]
		A pair of simplices $(\sigma^p,\tau^{p+1})$ is called an
		\emph{apparent pair} if
		\begin{enumerate}
			\item $\sigma$ is the last codimension-one face of $\tau$ in the
			filtration order, and
			\item $\tau$ is the first codimension-one coface of $\sigma$ in the filtration order.
		\end{enumerate}
		Equivalently, the entry corresponding to $(\sigma,\tau)$ in the
		filtration boundary matrix is nonzero, and all entries below it in the
		column of $\tau$ and all entries to its left in the row of $\sigma$ are
		zero.
	\end{define}
	
	\noindent 
	An important property of apparent pairs is that the collection of apparent pairs of a simplex-wise filtration forms a
	discrete gradient vector field (see \cite[Lemma~3.5]{Rip}).
	\smallskip
	\noindent
	For each $t$, let $A_t:=A(K_t)$ denote the collection of apparent pairs among the first $t$ simplices of the filtration. 
	We define the \emph{apparent
		Morse complexity profile} by \[
	\mathcal C_{\mathrm{app}}(K_\bullet)
	:=
	\bigl\{C_{\mathrm{app}}(K_t)
	\bigr\}_{t=0}^{N},
	\]
	where, $C_{\mathrm{app}}(K_t)$ denote the number of simplices left unmatched by the apparent-pair
	matching $A_t$.  Since $A_t$ is an acyclic discrete Morse matching on $K_t$, $
	M(K_t)\leq C_{\mathrm{app}}(K_t)$ for every $t$.  Thus, the apparent Morse complexity profile gives a directly computable
	upper bound for the Morse complexity profile.
	
	\smallskip
	\noindent
	We now study how the apparent Morse complexity profile changes under
	an adjacent transposition. For any discrete Morse matching $\mathcal V_t$ on $K_t$, we know that 
	\[
	|\operatorname{Crit}(\mathcal V_t)|
	=
	|K_t|-2|\mathcal V_t|.
	\]
	Consequently, if two complexes differ by at most $\delta$ simplices and
	their matchings differ in size by at most $r$, then their numbers of
	critical simplices differ by at most $\delta+2r$. We use this observation
	below for apparent-pair matchings.

	\begin{theorem}\label{thm:app}
		Let $K_\bullet$ and $L_\bullet$ be two simplex-wise filtrations of the
		same finite simplicial complex $K$, whose filtration orders differ by
		a single adjacent transposition of two simplices $u$ and $v$ at positions $i$ and $i+1$. Then,
		\[
		\left|
		C_{\mathrm{app}}(K_t)-C_{\mathrm{app}}(L_t)
		\right|
		\leq
		\begin{cases}
			0, & t<i,\\[2mm]
			2, & t=i,\\[2mm]
			4, & t\geq i+1.
		\end{cases}
		\]
	\end{theorem}
	\begin{proof}
		Note that if $\dim u\neq\dim v$, then interchanging $u$ and $v$ cannot
		affect either of the two conditions defining an apparent pair. Hence,
		the apparent-pair matching remains unchanged. We therefore assume that $ \dim u=\dim v. $
		
		\smallskip
		\noindent
		For $t<i$, the two prefix filtrations are identical. Hence, their
		apparent-pair matchings coincide and
		\[
		C_{\mathrm{app}}(K_t)=C_{\mathrm{app}}(L_t).
		\]
		
		\noindent 
		For $t=i$, the complexes $K_i$ and $L_i$ are obtained from the common
		complex $K_{i-1}=L_{i-1}$ by adding $u$ and $v$, respectively. Since
		the order of all simplices in $K_{i-1}=L_{i-1}$ is unchanged, an
		apparent pair consisting entirely of simplices in this common complex
		remains unchanged. Thus, any difference between the apparent pairs of $K_i$ and $L_i$
		must involve the newly added simplex. More precisely, an apparent pair $(a,u)$ in $K_i$ can be replaced by
		$(a,v)$ in $L_i$ only if $a$ is a codimension-one face of $v$.
		Similarly, an apparent pair $(u,b)$ can be replaced by $(v,b)$ only if
		$v$ is a codimension-one face of $b$. The replacement of one pair with another pair should not change the number of apparent pairs. However, if the pair is broken without the replacement or is newly formed then it should change the number. Since a
		simplex can belong to at most one apparent pair, the change in the
		number of apparent pairs is at most one. Therefore,
		\[
		\bigl||A(K_i)|-|A(L_i)|\bigr|\leq 1.
		\]
		Further, as $|K_i|=|L_i|$,
		\[
		\left|
		C_{\mathrm{app}}(K_i)-C_{\mathrm{app}}(L_i)
		\right|
		\leq 2.
		\]
		
		\noindent Now let $t\geq i+1$. Then, $K_t=L_t$ as simplicial complexes, and their
		induced filtration orders differ only in the relative order of $u$ and
		$v$. Therefore, every apparent
		pair containing neither $u$ nor $v$ remains unchanged. Every
		apparent pair that is lost or created must contain $u$ or $v$.
		Since the apparent pairs form a matching, each of $u$ and $v$ occurs
		in at most one apparent pair. Consequently,
		\[
		\bigl||A(K_t)|-|A(L_t)|\bigr|\leq 2.
		\]
		Finally, since $K_t=L_t$,
		\[
		\begin{aligned}
			\left|
			C_{\mathrm{app}}(K_t)-C_{\mathrm{app}}(L_t)
			\right|
			&=
			\left|
			\bigl(|K_t|-2|A(K_t)|\bigr)
			-
			\bigl(|L_t|-2|A(L_t)|\bigr)
			\right|\\
			&=
			2\bigl||A(K_t)|-|A(L_t)|\bigr|\\
			&\leq 4.
		\end{aligned}
		\]
		This proves the claim.
	\end{proof}
	\begin{cor}
		Let $K_\bullet$ and $L_\bullet$ be two simplex-wise filtrations of the
		same finite simplicial complex. Suppose that their filtration orders
		can be related by a sequence of $r$ adjacent transpositions, each
		between simplices of the same dimension. Then,
		$
		\sup_t
		\left|
		C_{\mathrm{app}}(K_t)-C_{\mathrm{app}}(L_t)
		\right|
		\leq 4r.
		$
	\end{cor}
	
	\begin{proof}
		Applying the preceding lemma successively to the $r$ transpositions and using the triangle inequality yields the result.
	\end{proof}
	\begin{example}[Sharpness of the one-transposition bound]
		\label{ex:apparent-sharp}
		
		Consider two simplex-wise filtrations $K_\bullet$ and $L_\bullet$ of
		the same simplicial complex, with filtration orders
		\[
		K_\bullet:\qquad
		d<a<c<e<cd<
		\underbrace{ac<ad}_{\text{transposed}}
		<ae<de<ade<acd,
		\]
		\[
		L_\bullet:\qquad
		d<a<c<e<cd<
		\underbrace{ad<ac}
		<ae<de<ade<acd.
		\]
		Thus, $L_\bullet$ is obtained from $K_\bullet$ by a single adjacent
		transposition of the two edges $ac$ and $ad$ at $i=6$.
		
		\smallskip
		\noindent For $t<6$, the filtrations agree, so
		$
		C_{\mathrm{app}}(K_t)=C_{\mathrm{app}}(L_t).
		$
		At $t=6$,
		\[
		A(K_6)=\{(c,cd)\},
		\qquad
		A(L_6)=\{(c,cd),(a,ad)\},
		\]
		and hence
		\[
		|C_{\mathrm{app}}(K_6)-C_{\mathrm{app}}(L_6)|=2.
		\]
		
		\noindent At the final level $t=11$, we have
		\begin{align*}
			&A(K_{11})=\{(c,cd),(e,ae),(de,ade)\},\qquad 
			A(L_{11})
			&=\{(c,cd),(a,ad),(e,ae),(de,ade),(ac,acd)\}.
		\end{align*}
		Therefore,
		$
		C_{\mathrm{app}}(K_{11})=5$ and $
		C_{\mathrm{app}}(L_{11})=1,
		$
		so
		\[
		|C_{\mathrm{app}}(K_{11})-C_{\mathrm{app}}(L_{11})|=4.
		\]
	\end{example}
	
	\smallskip
	\noindent
	\textbf{2. Last-facet approximation.}
	We next introduce a second computable approximation to the Morse
	complexity profile. By retaining only the last-facet condition of
	apparent pairs, this construction can yield a sharper upper bound on
	the minimal Morse number and a stronger point-wise stability bound under
	adjacent transpositions.
	
	\smallskip
	\noindent
	For every non-vertex simplex $\tau$, let $\sigma_{\tau}$ denote the
	last codimension-one face of $\tau$ in the filtration order. Define
	the \emph{last-facet graph} $F_K$ by defining set of vertices and edges by
	\[
	V(F_K)=K,
	\qquad
	E(F_K)=\bigl\{\{\sigma_{\tau},\tau\}:
	\tau\in K,\ \dim\tau>0\bigr\}.
	\]
	A matching in a graph is a collection of edges no two of
	which share a vertex.
	\begin{lemma}\label{lem:forest}
		The graph $F_K$ is a forest. Moreover, every graph matching in $F_K$ 
		defines an acyclic matching on $K$.
	\end{lemma}
	
	\begin{proof}
		Suppose that $F_K$ contains a cycle, and choose a simplex $\tau$ of
		maximal dimension on the cycle. Then, both its neighbors on the cycle
		must be codimension-one faces of $\tau$,  contradicting the fact that $\tau$ is joined only to its last facet
		$\sigma_{\tau}$. Hence, $F_K$ is a forest.
		
		\smallskip
		\noindent Now let $\mathcal V$ be a graph matching in $F_K$. If $\mathcal V$ contained a closed
		gradient path, then at every matched pair $(\sigma_{\tau},\tau)$ the
		path would leave $\tau$ through a different codimension-one face
		$\sigma$. Since $\sigma_{\tau}$ is the last face of $\tau$, we have
		$\sigma<\sigma_{\tau}$. Consequently, along any gradient path the matched lower-dimensional
		simplices strictly decrease in the filtration order. Therefore,
		such a path cannot be closed, and $\mathcal V$ is acyclic.
	\end{proof}
	\noindent Note that every apparent pair $(\sigma,\tau)$ is an edge of $F_K$.
	Moreover, the apparent pairs themselves form a graph matching in
	$F_K$. Therefore, if $\nu(F_K)$ denotes the maximum size of a graph
	matching in $F_K$, then
	\[
	\nu(F_K)\geq |A(K)|.
	\]
	For each filtration level $t$, let $F_{K_t}$ be the last-facet forest
	of $K_t$ with respect to the induced filtration order, and let
	$\nu(F_{K_t})$ denote its maximum matching number. By Lemma \ref{lem:forest}, every such matching determines an acyclic discrete Morse
	matching on $K_t$. For each $t$, set
	\[
	C_{\mathrm{lf}}(K_t):=|K_t|-2\nu(F_{K_t}).
	\]
	Thus, $C_{\mathrm{lf}}(K_t)$ is the number of critical simplices in any
	maximum matching of the last-facet forest $F_{K_t}$.
	These values determine the \emph{last-facet Morse complexity profile}
	\[
	\mathcal C_{\mathrm{lf}}(K_\bullet)
	:=
	\bigl\{C_{\mathrm{lf}}(K_t)\bigr\}_{t=0}^{N}.
	\]
	Moreover,
	\[
	M(K_t)\leq C_{\mathrm{lf}}(K_t)
	\leq C_{\mathrm{app}}(K_t)
	\qquad\text{for every }t.
	\]
	\smallskip
	\noindent
	We now show that, in addition to providing a sharper upper bound, the last-facet profile also satisfies a stronger point-wise stability bound under adjacent transpositions.
	
	\begin{theorem}\label{thm:last-facet}
		Let $K_\bullet$ and $L_\bullet$ be two simplex-wise filtrations as in Theorem~\ref{thm:app}.
		Then,
		\[
		\sup_t
		\left|
		C_{\mathrm{lf}}(K_t)-C_{\mathrm{lf}}(L_t)
		\right|
		\leq 2.
		\]
	\end{theorem}
	\begin{proof}
		Since the transposition exchanges the simplices $u$ and $v$
		at positions $i$ and $i+1$, for $t<i$, we have $K_t=L_t$, with the same induced order. Hence,
		\[
		F_{K_t}=F_{L_t},
		\]
		and the two profiles agree.
		
		\smallskip
		\noindent At $t=i$, both $K_i$ and $L_i$ are obtained from the common complex
		$K_{i-1}=L_{i-1}$ by adding one simplex, namely $u$ and $v$,
		respectively. In either case, the corresponding last-facet forest is
		obtained from the common forest by adding one new vertex and at most
		one edge incident to it. Adding such a vertex can increase the maximum
		matching size by at most one. Therefore,
		\[
		\left|\nu(F_{K_i})-\nu(F_{L_i})\right|\leq 1.
		\]
		
		\smallskip
		\noindent Now let $t\geq i+1$. Then, $K_t=L_t$ as simplicial complexes, and the
		two orders differ only in the relative positions of $u$ and $v$.
		Thus, the
		last-facet choice can change only for a simplex having both $u$ and
		$v$ as facets. Since two distinct simplices of the same dimension can
		be facets of at most one common simplex, at most one last-facet edge
		can change. Hence, $F_{K_t}$ and $F_{L_t}$ differ by at most one edge
		replacement.
		
		\smallskip
		\noindent
		Removing the replaced edge from a maximum matching leaves a matching
		of size at least one less in the other forest. Interchanging the roles
		of the two forests therefore gives
		\[
		\left|\nu(F_{K_t})-\nu(F_{L_t})\right|\leq 1
		\]
		and hence,
		\[
		|C_{\mathrm{lf}}(K_t)-C_{\mathrm{lf}}(L_t)|\leq2.
		\]
	\end{proof}
	\smallskip
	\noindent
	\textbf{3. Approximation via a global Morse matching.}
	The preceding constructions compute a Morse matching independently at
	each filtration level. We now consider a different approach: choose a
	single discrete Morse matching on the full complex that is independent of any filtration on it and restrict it
	along the filtration.
	Let $\mathcal V$ be a discrete Morse matching on $K$. For a simplex-wise
	filtration
	\[
	K_\bullet:K_0\subset K_1\subset\cdots\subset K_N=K,
	\]
	define
	\[
	\mathcal V_t
	:=
	\{(\sigma,\tau)\in\mathcal V:\tau\in K_t\}.
	\]
	Since $\sigma$ is a face of $\tau$, $\tau\in K_t$ implies
	$\sigma\in K_t$. Hence, $\mathcal V_t$ is precisely the restriction of
	$\mathcal V$ to $K_t$. Clearly, $\mathcal V_t$ is acyclic, since a
	closed gradient path in $\mathcal V_t$ would also be a closed gradient
	path in $\mathcal V$. Let $C_{\mathcal V}(K_t)$ denote the number of critical simplices at
	level $K_t$, then
	\[
	M(K_t)\leq C_{\mathcal V}(K_t).
	\]
	\begin{theorem}\label{thm:global-matching-stability}
		Let $K_\bullet$ and $L_\bullet$ be as in
		Theorem~\ref{thm:app}, and let $\mathcal V$
		be a fixed and filtration independent discrete Morse matching on $K$. Then,
		\[
		\sup_t
		\left|C_{\mathcal V}(K_t)-C_{\mathcal V}(L_t)\right|
		\leq 2.
		\]
	\end{theorem}
	
	\begin{proof}
		The restricted matchings coincide for $t<i$ and $t\geq i+1$, so only
		$t=i$ can differ. Since each of the transposed simplices $u$ and $v$
		belongs to at most one pair of $\mathcal V$,
		\[
		\bigl||\mathcal V_i^K|-|\mathcal V_i^L|\bigr|\leq1.
		\]
		As $|K_i|=|L_i|$, this gives
		\[
		|C_{\mathcal V}(K_i)-C_{\mathcal V}(L_i)|\leq2,
		\]
		and the result follows.
	\end{proof}
	
	\smallskip
	\noindent\textbf{Global versus order-independent matchings.}
	An apparent-pair matching may be computed once on the whole complex
	$K$. Its restriction to any sublevel complex $K_t$ agrees exactly with
	the apparent-pair matching computed directly on $K_t$ using the induced
	order. Nevertheless, this global matching is still order-dependent and
	may change when the filtration order changes.
	
	\smallskip
	\noindent The last-facet forest has the same restriction property: its restriction
	to $K_t$ is precisely the last-facet forest constructed on $K_t$.
	However, a maximum matching of the full forest need not restrict to a
	maximum matching of the restricted forest. The level-wise construction
	therefore maintains, or recomputes, a maximum matching at each $K_t$.
	The individual pairs may change, but their number is determined by the
	maximum matching size of the restricted forest.
	
	\smallskip
	\noindent This flexibility explains the different stability bounds. Under one
	adjacent transposition, the apparent-pair profile changes by at most
	$4$, whereas the levelwise last-facet profile changes by at most $2$.
	The latter agrees with the bound obtained by restricting a fixed,
	order-independent global Morse matching.

	\smallskip
	\noindent\textbf{4. Greedy Morse approximation.}
	We finally consider the greedy acyclic matching algorithm. Unlike the  constructions above, the greedy procedure allows more
	flexibility in choosing matching pairs and may therefore produce a
	better Morse matching. However, this additional flexibility comes at the cost of stability:
	its critical-simplex count need not satisfy any uniform point-wise
	bound under a single adjacent transposition.
	
	\smallskip
	\noindent Let $K$ be a finite simplicial complex, and let $\mathcal{K}$ denote its face poset. We use the Hasse-diagram convention of
	Section~\ref{S2}, in which matched edges are reversed. Recall that a
	matching is acyclic precisely when the modified Hasse diagram contains no directed cycles \cite{CH}.
	
	\smallskip
	\noindent
	We first describe the greedy procedure. Fix once and for all a total
	order $\prec$ on the simplices to break ties. At each filtration level,
	the newly appearing simplices are processed by increasing dimension
	and, within each dimension, according to $\prec$. When a simplex
	$\sigma$ is processed, its codimension-one faces are examined according
	to $\prec$. The simplex $\sigma$ is paired with the first currently
	unmatched codimension-one face $\tau$ for which adding $(\tau,\sigma)$
	preserves acyclicity. If no such face exists, no pair is added at this
	step. Each simplex is processed only once, at the filtration level at
	which it appears, although it may later be matched as a face of a newly
	appearing simplex.

	\smallskip
	\noindent
	To determine whether $(\tau,\sigma)$ preserves acyclicity, we
	temporarily remove the edge $\sigma\to\tau$ and check whether there is
	a directed path $\sigma\leadsto\tau$ in the current modified Hasse
	diagram. If no such path exists, the edge is reversed and
	$(\tau,\sigma)$ is added to the matching. Otherwise, the edge is
	restored and the next available face is considered. After all simplices up to level $t$ have been processed, the resulting
	acyclic matching on $K_t$ is denoted by $V_t$. The unmatched simplices
	of $V_t$ are its critical simplices, and their number is recorded as
	$C_{\textup{gr}}(K_t)$. The procedure is described precisely in
	Algorithm~\ref{alg:greedy-acyclic-matching}.
	
	\begin{algorithm}[H]
		\caption{Greedy acyclic matching along a filtration}
		\label{alg:greedy-acyclic-matching}
		
		\Require{
			A filtration
			$K_\bullet:K_0\subseteq K_1\subseteq\cdots\subseteq K_n$
			and a fixed total order $\prec$ on the simplices
		}
		\Ensure{
			An acyclic matching $V_t$ on each $K_t$ and the corresponding
			number $C_{\textup{gr}}(K_t)$ of critical simplices
		}
		
		$V\gets\varnothing$\;
		$C\gets 0$\;
		$H\gets\varnothing$\tcp*{Current directed Hasse diagram}
		
		\For{$t=0,\ldots,n$}{
			$S_t\gets K_t\setminus K_{t-1}$\tcp*{$K_{-1}=\varnothing$}
			
			Order the simplices of $S_t$ first by increasing dimension
			and then according to $\prec$\;
			
			\ForEach{$\sigma\in S_t$ in this order}{
				Insert $\sigma$ and its downward Hasse edges into $H$\;
				mark $\sigma$ as unmatched\;
				$C\gets C+1$\;
				
				\ForEach{codimension-one face $\tau<\sigma$
					in the order induced by $\prec$}{
					\If{$\tau$ is unmatched}{
						Temporarily remove the edge $\sigma\to\tau$ from $H$\;
						
						\eIf{there is no directed path
							$\sigma\leadsto\tau$ in $H$}{
							Insert the reversed edge $\tau\to\sigma$ into $H$\;
							$V\gets V\cup\{(\tau,\sigma)\}$\;
							mark $\tau$ and $\sigma$ as matched\;
							$C\gets C-2$\;
							\textbf{break}\;
						}{
							Restore the edge $\sigma\to\tau$ in $H$\;
						}
					}
				}
			}
			
			$V_t\gets V$\;
			$C_{\textup{gr}}(K_t)\gets C$\;
		}
		
		\Return $\{(V_t,C_{\textup{gr}}(K_t))\}_{t=0}^{n}$\;
	\end{algorithm}
	
	\begin{theorem}\label{thm:greedy-algo}
		Let
		$
		K_\bullet: K_0 \subseteq K_1 \subseteq \cdots \subseteq K_n
		$
		be a finite filtration of simplicial complexes of maximal dimension $d$, and let
		$N$
		be the total number of simplices that appear in the filtration.
		Then, a greedy acyclic matching algorithm (see  Algorithm~\ref{alg:greedy-acyclic-matching}) computes
		for every $t$ a number $C_{\textup{gr}}(K_t)$ equal to the number
		of critical simplices in the constructed matching. Moreover, the total
		additional time complexity of this computation is $O(d^2N^2)$ beyond the cost
		of constructing the filtration.
	\end{theorem}
	\begin{proof}
		\textbf{\emph{Termination and correctness.}}
		The algorithm terminates since the filtration is finite, each simplex
		is processed exactly once, and only finitely many codimension-one faces
		are examined for each simplex.
		
		\smallskip\noindent For correctness, a pair $(\tau,\sigma)$ is added only when both
		simplices are unmatched, so $V$ remains a matching. Since the current
		modified Hasse diagram is acyclic, reversing $\sigma\to\tau$ creates a
		directed cycle if and only if, after deleting this edge, there is already
		a directed path $\sigma\leadsto\tau$. The algorithm accepts the pair
		precisely when no such path exists. Hence, the modified Hasse diagram
		remains acyclic throughout. Therefore, at every filtration level $t$,
		$V_t$ is an acyclic matching on $K_t$, and $C_{\textup{gr}}(K_t)$ is exactly the
		number of critical simplices of $V_t$.

		\smallskip
		\noindent
		\textbf{\emph{Complexity.}} Each simplex of the filtration enters exactly once, namely at the first index $t$ for which it belongs to $K_t$. Hence, $N$ equals the total number of simplices processed by the algorithm.

		\smallskip
		\noindent
		\emph{Number of candidate faces.} Since a $p$-simplex has exactly $p+1$ codimension-one faces and $p+1\le d+1$, at most $O(d)$ candidate faces are examined.

		\smallskip\noindent \emph{Size of the Hasse diagram.} The Hasse diagram contains one vertex for each simplex, and hence at
		most $N$ vertices. Moreover, every $p$-simplex contributes exactly
		$p+1$ downward Hasse edges. Therefore,
		\[
		|E|
		=
		\sum_{\sigma\in K}
		\#\{\text{codimension-one faces of }\sigma\}
		\leq
		(d+1)N,
		\]
		so that $|E|=O(dN)$. Reversing matched edges does not change the number
		of edges. Hence, a breadth-first or depth-first search in the modified
		Hasse diagram requires
		\[
		O(|V|+|E|)=O(dN)
		\]
		time in the worst case. Thus, the acyclicity test for one candidate pair
		costs $O(dN)$.
		
		\smallskip\noindent\emph{Work per simplex.}
		For a simplex $\sigma$, at most $O(d)$ candidate faces are tested, and
		each acyclicity test costs $O(dN)$. Hence, the worst-case work required
		for processing one simplex is
		$$
		O(d)\cdot O(dN)=O(d^2N).
		$$
		
		\smallskip
		\noindent
		\emph{Critical simplex count.}
		At every stage, $C_{\textup{gr}}(K_t)$ is the number of unmatched simplices in the
		current matching on $K_t$. When a simplex enters the filtration, this
		count and the matching status of the simplices can be updated in
		constant time whenever a new pair is created. Since each simplex enters
		once and can belong to at most one matched pair, these updates contribute
		only $O(N)$ additional work in total.
		
		\smallskip
		\noindent
		Since $N$ simplices are processed in total and each requires $O(d^2N)$ work, the overall additional running time is
		$$
		O(N)\cdot O(d^2N)=O(d^2N^2).
		$$
		In particular, if the dimension $d$ is bounded, the complexity reduces to
		$$
		O(N^2).
		$$
		
	\end{proof}

	\begin{remark}
		The value $C_{\textup{gr}}(K_t)$ obtained by the greedy algorithm equals the
		number of critical simplices in the constructed matching. In general, it
		is only an upper bound to the minimal Morse number $M(K_t)$, since
		computing an optimal discrete Morse matching is NP-hard.
	\end{remark}
	\begin{prop}[Instability of the greedy matching]
		\label{prop:greedy-instability}
		For every $n\geq1$, there is a two-dimensional simplicial complex
		$K_n$ with $18n+16$ non-empty simplices and two face-compatible total
		orders, differing by one transposition of adjacent edges, such that the
		numbers of critical simplices in the resulting greedy matchings differ
		by $2n$. Hence, no uniform stability bound is possible under adjacent
		transpositions.
	\end{prop}
	
	\begin{proof}
		Fix $n\geq 1$. Let $K_n$ be the simplicial complex with distinct
		vertices
		\[
		\{0,1,2,\omega,z,v_1,\ldots,v_n,x_1,\ldots,x_n,y_1,\ldots,y_n\},
		\] and set $v_0:=1$. For simplicity, we denote a simplex by concatenating its vertices, for instance, $012$ denotes the triangle $\{0,1,2\}$.

		\smallskip
		\noindent
		The complex $K_n$ is generated by the triangles
		\[
		012,\qquad 02\omega,
		\]
		and, for $1\leq i\leq n$,
		\[
		0v_{i-1}v_i,\qquad v_{i-1}v_ix_i,\qquad v_ix_iy_i,
		\qquad 02x_i,\qquad 0x_iy_i,
		\]
		together with the maximal edges
		\[
		z0,\ z1,\ z2,\ z\omega,\ zv_1,\ldots,zv_n.
		\]
		Thus, $K_n$ has $3n+5$ vertices, $10n+9$ edges, and $5n+2$
		triangle, and hence $18n+16$ non-empty simplices in total.
		
		\smallskip
		\noindent 
		We now define two total orders $\prec$ and $\prec'$ on the simplices
		of $K_n$. In both orders, all vertices precede all edges, and all edges
		precede all triangles.
		Place every vertex before $z$, and process the edges incident to $z$
		first. Each edge $zu$ pairs with $u$, matching\[0,1,2,\omega,v_1,\ldots,v_n.\]
		giving the pairs
		\[
		(0,z0),(1,z1),(2,z2),(\omega,z\omega),
		(v_1,zv_1),\ldots,(v_n,zv_n).
		\]
		
		\smallskip
		\noindent Next, process the edges $2x_i$ and $0y_i$. Since $2$ and $0$ are
		already matched, these edges pair with $x_i$ and $y_i$, respectively.
		We thus obtain $3n+4$ vertex--edge pairs, and every vertex except $z$ is
		matched.
		
		\smallskip
		\noindent Let
		\[
		a=02,\qquad b=01,\qquad p_i=0v_i,\qquad
		\ell_i=v_{i-1}v_i,\qquad s_i=v_ix_i,\qquad
		e_i=x_iy_i,\qquad d_i=0x_i,
		\]
		so that $p_0=b$. After the edges above, place $a$ and $b$
		consecutively and place all remaining edges afterwards. Choose their
		order so that
		\[
		p_{i-1}\prec p_i\prec\ell_i,\qquad
		\ell_i\prec v_{i-1}x_i,\qquad
		\ell_i\prec s_i,\qquad
		s_i\prec v_iy_i,\qquad
		s_i\prec e_i,
		\]
		and $0\omega\prec2\omega$. 
		
		\smallskip
		\noindent The two orders differ only in the positions of the consecutive edges
		$a$ and $b$: in $\prec$ we have $a\prec b$, whereas in $\prec'$ we
		have $b\prec' a$.
		\smallskip
		\noindent Process the triangles in the order
		\[
		U,W,A_1,B_1,C_1,\ldots,A_n,B_n,C_n,
		D_1,\ldots,D_n,E_1,\ldots,E_n,
		\]
		where
		\[
		\begin{aligned}
			U&=012,                 &W&=02\omega,\\
			A_i&=0v_{i-1}v_i,       &B_i&=v_{i-1}v_ix_i,\\
			C_i&=v_ix_iy_i,         &D_i&=02x_i,\\
			E_i&=0x_iy_i.
		\end{aligned}
		\]
		The two orders produce the following triangle--edge pairs listed in the table below:
		\[
		\begin{array}{c|ccccccc}
			&U&W&A_i&B_i&C_i&D_i&E_i\\ \hline
			a\prec b
			&a&0\omega&p_{i-1}&\ell_i&s_i&d_i&\text{none}\\
			b\prec' a
			&b&a&p_i&\ell_i&s_i&d_i&e_i .
		\end{array}
		\]
		
		\smallskip
		\noindent We verify the only nontrivial acyclicity decision. Recall that adding
		a pair $(\tau,\sigma)$ reverses the incidence arrow
		$\sigma\to\tau$. It creates a cycle precisely when, after deleting
		this arrow, there is an alternative directed path from $\sigma$ to
		$\tau$. All pairs in the table through $D_i$ are accepted: in each
		case the possible alternative paths end at edges not yet paired with
		triangles. For $(d_i,D_i)$, the only other triangle containing $d_i$
		is the later triangle $E_i$. In the second order, the only other
		triangle containing $p_i$ is $A_{i+1}$ when $i<n$, and it has not yet
		been processed.
		
		\smallskip
		\noindent Now suppose that $a\prec b$. When $E_i$ is processed, its facets
		$d_i$ and $0y_i$ are already matched, so its only possible partner is
		$e_i$. Adding $(e_i,E_i)$ would create the directed cycle
		\[
		\begin{aligned}
			e_i\to E_i\to d_i\to D_i\to a\to U\to p_0
			\to A_1\to p_1\to A_2\to\cdots\to p_{i-1}\to A_i\\
			\to\ell_i\to B_i\to s_i\to C_i\to e_i,
		\end{aligned}
		\]
		with the evident truncation when $i=1$. Hence, every $E_i$ remains
		critical.

		\smallskip
		\noindent For the order $b\prec' a$, the edge $b$ pairs with $U$, while $a$
		pairs with $W$. Any alternative path leaving $E_i$ either reaches the
		edge $0y_i$, or follows
		\[
		E_i\to d_i\to D_i.
		\]
		From $D_i$, it reaches either $2x_i$ or
		\[
		a\to W\to 0\omega,\ 2\omega.
		\]
		None of these edges is paired with a triangle. Thus no path returns to
		$e_i$, and every pair $(e_i,E_i)$ is accepted.
		
		\smallskip
		\noindent Both matchings contain the $3n+4$ vertex--edge pairs and the $4n+2$
		triangle--edge pairs involving $U,W,A_i,B_i,C_i$, and $D_i$. The
		second matching contains the additional $n$ pairs $(e_i,E_i)$.
		Therefore, the two matchings contain
		\[
		7n+6\qquad\text{and}\qquad8n+6
		\]
		pairs, respectively. Their critical counts are consequently
		\[
		(18n+16)-2(7n+6)=4n+4
		\]
		and
		\[
		(18n+16)-2(8n+6)=2n+4.
		\]
		Thus, their numbers of critical simplices differ by $2n$.
		
		\smallskip
		\noindent
		Consider the one-level filtration $K_0=K_n$. The two orders are
		valid tie-breaking orders for the same filtration, and the critical
		counts at the full sublevel complex differ by $2n$. Since
		\[
		2n=\frac{|K_n|-16}{9},
		\]
		the instability grows linearly with the size of the complex.
	\end{proof}
	
	\section{From Morse complexity to Whitehead torsion}\label{S5}
	
	\noindent
	In Section \ref{S3}, we introduced the Morse complexity profile as a tool to study certain simple-homotopy phenomena along a filtration. 
	While persistent homology records the evolution of homological information, the Morse complexity profile detects changes in the combinatorial complexity of discrete Morse functions that remain invisible to homology. 
	Classical simple-homotopy theory provides an obstruction associated to a homotopy equivalence, namely the \emph{Whitehead torsion}. 
	This invariant distinguishes homotopy equivalences from simple-homotopy equivalences and therefore lies naturally beyond both homology and discrete Morse theory.
	
	\smallskip
	\noindent
	The goal of this section is to extend this classical invariant to the persistent
	setting. 
	Rather than assigning Whitehead torsion to an isolated homotopy equivalence,
	we seek to associate compatible torsion data to the inclusion maps occurring in a
	filtration. 
	The resulting invariant refines persistent homology by recording
	simple-homotopy information that is invisible at the level of homology.

	\smallskip
	\noindent
	A key observation is that Whitehead torsion is fundamentally a homotopy invariant.
	Indeed, if two homotopy equivalences are connected through a homotopy consisting of
	homotopy equivalences, then they have identical Whitehead torsion. 
	Consequently, Whitehead torsion should naturally be regarded not merely as an invariant of individual maps, but as a locally constant invariant on the space of homotopy equivalences between two finite CW complexes. 
	The persistent construction developed below is obtained by evaluating this classical invariant on the distinguished inclusion maps arising from a filtration.

	\begin{prop}
		Let $X$ and $Y$ be finite connected CW complexes, and let
		$\operatorname{Map}(X,Y)$ denote the space of continuous maps from
		$X$ to $Y$, equipped with the compact-open topology. Let
		\[
		\mathrm{Eq}(X,Y)\subseteq \operatorname{Map}(X,Y)
		\]
		denote the subspace consisting of homotopy equivalences, equipped with
		the induced subspace topology.
		Then, the Whitehead torsion defines a locally constant map
		\[
		\tau:\mathrm{Eq}(X,Y)\longrightarrow \mathrm{Wh}(\pi_1(Y)).
		\]
	\end{prop}
	\begin{proof}
		Whitehead torsion is invariant under homotopy of homotopy equivalences. Therefore,
		if two points of $\mathrm{Eq}(X,Y)$ lie in the same path component, they determine the same
		Whitehead torsion. Hence, $\tau$ is constant on each path component of $\mathrm{Eq}(X,Y)$, and
		therefore locally constant.
	\end{proof}
	
	\begin{remark}
		The preceding proposition shows that Whitehead torsion is naturally associated with
		the mapping space of homotopy equivalences rather than merely with isolated maps.
		The persistent invariant introduced below is obtained by evaluating this locally
		constant mapping-space invariant on the canonical inclusion maps of a filtration.
	\end{remark}

	\smallskip
	\noindent
	Let $G$ be a group and $\mathrm{Wh}(G)$ be the Whitehead (abelian) group of $G$ (see \cite[\S 2]{ALQ}). Let $\widehat{\mathrm{Wh}(G)}$ denote the category consisting of a single object $\{\ast\}$ and a set of morphisms $\textrm{Hom}_{\widehat{\mathrm{Wh}(G)}}(\ast,\ast)=\mathrm{Wh}(G).$
	The composition is given by the group operation of $\mathrm{Wh}(G)$, $a\circ b:=a+b,$
	and the identity morphism is the zero element $0\in\mathrm{Wh}(G).$
	Since $\mathrm{Wh}(G)$ is a group, every morphism is invertible.
	Hence, $\widehat{\mathrm{Wh}(G)}$ is a one-object groupoid.

	\begin{define}[Torsion-eligible window]\label{def:tor-eligible}
		Let $\{K_t\}_{t\in I}$ be a filtration of finite connected simplicial complexes over an index set $I$. A subset $J \subseteq I$ is called a \emph{torsion-eligible window} if for all $s \le t$ in $J$, the inclusion $\iota_{s,t}:K_s \hookrightarrow K_t$ is a homotopy equivalence.
		
		\smallskip 
		\noindent
		We refer to the restricted filtration $\{K_t\}_{t\in J}$ as the
		\emph{torsion-eligible filtration} associated with $J$.
	\end{define}
	
	\begin{remark}
		Classically, the Whitehead torsion of a homotopy equivalence
		$f:X\rightarrow Y$ takes values in the Whitehead group $\mathrm{Wh}(\pi_1(Y))$, so that different maps naturally take values in different groups. 
		Consequently, Whitehead torsions cannot, in general, be compared across an arbitrary filtration. 
		The role of a torsion-eligible window is precisely to identify these varying Whitehead groups with a single Whitehead group, thereby allowing torsion values associated to
		different inclusion maps to be compared and composed.
	\end{remark}
	
	\textit{Note:} Given an inclusion map $\iota_{s,t} : K_s \hookrightarrow K_t$, we use the notation $(\iota_{s,t})_{\ast}: \pi_1(K_s) \longrightarrow \pi_1(K_t)$ for the induced map on fundamental groups  and 
	$$(\iota_{s,t})_*^{\mathrm{Wh}}: \mathrm{Wh}(\pi_1(K_s)) \longrightarrow \mathrm{Wh}(\pi_1(K_t))$$ 
	for the map induced on the Whitehead groups.
	
	\begin{lemma}[Compatible identification of fundamental groups]\label{lem:compat}
		Let $K_\bullet=\{K_t\}_{t\in I}$
		be a filtration of finite connected simplicial complexes, and let
		$J\subseteq I$ be a torsion-eligible window. Let $t_0=\min J$ and set
		\[
		G:=\pi_1(K_{t_0}).
		\]
		After choosing a basepoint in $K_{t_0}$, the groups $\pi_1(K_t)$,
		$t\in J$, admit compatible identifications
		\[
		\alpha_t:\pi_1(K_t)\xrightarrow{\cong}G
		\]
		such that, for every $s\leq t$ in $J$,
		\begin{equation}\label{eq:compat}
			\alpha_t\circ(\iota_{s,t})_*=\alpha_s.
		\end{equation}
		Consequently, after applying the induced maps on Whitehead groups, the maps
		$$
		(\iota_{s,t})_*^{\mathrm{Wh}}:\mathrm{Wh}(\pi_1(K_s))
		\longrightarrow \mathrm{Wh}(\pi_1(K_t))
		$$
		are identified with the identity map on $\mathrm{Wh}(G)$.
	\end{lemma}
	
	\begin{proof}
		Since $J$ is torsion-eligible, every inclusion
		$\iota_{s,t}:K_s\hookrightarrow K_t$ is a homotopy equivalence.
		In particular, $(\iota_{t_0,t})_*$ is an isomorphism for every
		$t\in J$. Define
		$$
		\alpha_t=(\iota_{t_0,t})_*^{-1}.
		$$
		For $t_0\leq s\leq t$, the relation
		$ \iota_{t_0,t}\simeq \iota_{s,t}\circ\iota_{t_0,s} $
		gives
		$$
		(\iota_{t_0,t})_* = (\iota_{s,t})_* \circ (\iota_{t_0,s})_*. $$
		Taking inverses yields $ \alpha_t\circ(\iota_{s,t})_*=\alpha_s$.
		Thus, the identifications are compatible with the filtration maps.
		Applying the Whitehead group functor gives the final assertion.
	\end{proof}
	
	\noindent \noindent
	For the remainder of this section, we fix a torsion-eligible window
	$J\subseteq I$. Let $t_0=\min J$ and set
	\[
	G:=\pi_1(K_{t_0}).
	\]
	We use the compatible identifications
	$\alpha_t:\pi_1(K_t)\xrightarrow{\cong}G$ for all $t\in J$
	provided by the preceding lemma. Moreover, whenever several torsion-eligible
	filtrations and level-wise homotopy equivalences are under consideration,
	we choose the identifications of their fundamental groups compatibly
	with the relevant filtration maps and homotopy equivalences.\\

	\smallskip
	\noindent
	For $s\leq t$ in $J$, we now define  
	\begin{equation}\label{eq:tor}
		\tau_{s,t} := \alpha_{t*}^{\mathrm{Wh}} \bigl(\tau(\iota_{s,t})\bigr) \in \mathrm{Wh}(G),
	\end{equation}
	where $ \alpha_{t*}^{\mathrm{Wh}}:\mathrm{Wh}(\pi_1(K_t))\xrightarrow{\cong}\mathrm{Wh}(G) $
	is the isomorphism induced by $\alpha_t$.
	The resulting family
	\[\{\tau_{s,t}\}_{s\leq t,\;s,t\in J}
	\]
	is called the \emph{persistent Whitehead torsion assignment} on $J$.
	
	\begin{prop}[Persistent composition law]
		Let $\{K_t\}_{t\in J}$ be a torsion-eligible filtration of finite connected simplicial complexes.
		For every $s\le t\le u$ in $J$ we have,
		\begin{equation}\label{eq:cocyle0}
			\tau_{s,u} = \tau_{t,u} + \tau_{s,t} \quad \text{for all } s \le t \le u \text{ in } J.
		\end{equation}
	\end{prop}
	
	\begin{proof}
		Since $\iota_{s,u} = \iota_{t,u} \circ \iota_{s,t}$,
		the composition formula for Whitehead torsion (see \cite[Statement 22.4]{Coh}) gives
		\begin{equation}\label{torsion composition}
			\tau(\iota_{s,u})= \tau(\iota_{t,u}) + (\iota_{t,u})_{\ast}^{\mathrm{Wh}}(\tau(\iota_{s,t})),
		\end{equation}
		where $(\iota_{t,u})_{\ast}^{\mathrm{Wh}} : \mathrm{Wh}(\pi_1(K_t)) \to \mathrm{Wh}(\pi_1(K_u))$ is the map induced by $\iota_{t,u}$.
		Using the compatible identifications
		$\alpha_t:\pi_1(K_t)\xrightarrow{\cong}G$ for all $t\in J$ and by the functoriality of the Whitehead group
		\cite[\S 6]{Coh}, we obtain
		
		\begin{align*}
			\alpha_{u*}^{\mathrm{Wh}}(\tau(\iota_{s,u}))&=  \alpha_{u*}^{\mathrm{Wh}}(\tau(\iota_{t,u})) +  \alpha_{u*}^{\mathrm{Wh}}((\iota_{t,u})_{\ast}^{\mathrm{Wh}}(\tau(\iota_{s,t})))\\\notag
			\tau_{s,u}&=\tau_{t,u}+ \alpha_{t*}^{\mathrm{Wh}}(\tau(\iota_{s,t}))\tag{by \ref{eq:compat}}\\
			\tau_{s,u} &=\tau_{t,u}+ \tau_{s,u}\notag\qedhere
		\end{align*}
	\end{proof}
	
	\begin{define}[Persistent torsion cocycle]\label{D5.2}
		Let $\{K_t\}_{t\in J}$ be a torsion-eligible filtration of finite connected simplicial complexes.
		The assignment $(s \le t) \mapsto \tau_{s,t}$ is a \emph{$1$-cocycle} on the poset category $P(J)$ of the $J$ with values in $\mathrm{Wh}(G)$.
		We refer to $\tau$ as the \emph{persistent torsion cocycle} of the filtration on $J$. The cocycle vanishes iff every inclusion is a simple-homotopy equivalence.    
		
	\end{define}
	
	\noindent The preceding proposition shows that the persistent torsion assignments satisfy the same additive compatibility as classical Whitehead torsion. 
	It is therefore natural to package these assignments as a functor.
	
	\begin{define}[Persistent Whitehead torsion]\label{def:PWT}
		Let $\{K_t\}_{t\in J}$ be a torsion-eligible filtration of finite connected simplicial complexes. We define the \emph{persistent Whitehead torsion} of $K_\bullet$ on $J$ as the functor
		\[
		\matholdcal{T}(K_\bullet):P(J)\longrightarrow \widehat{\mathrm{Wh}(G)},\] which sends every object $t\longmapsto \ast$ and every morphism $
		(s\le t)
		\longmapsto
		\tau_{s,t}:\ast\to\ast.$
	\end{define}
	\noindent Clearly, $
	\tau_{t,t}
	=
	\alpha_{t*}^{\mathrm{Wh}} \bigl(\tau(\operatorname{id}_{K_t})\bigr)
	=
	0,$ which equals the identity morphism of
	$\widehat{\mathrm{Wh}(G)}$. Further, if $ s\le t \le u$, then $\iota_{s,u} = \iota_{t,u} \circ \iota_{s,t}$. It now follows from equation \eqref{eq:cocyle0} that
	\[
	\tau_{s,u}
	=
	\tau_{s,t}
	+
	\tau_{t,u}.
	\]
	Hence,
	\[
	\matholdcal T(s\le u)
	=
	\matholdcal T(t\le u)
	\circ
	\matholdcal T(s\le t),
	\]
	Therefore,
	$\matholdcal{T}(K_\bullet)$
	is a functor.

	\begin{remark}[Beyond torsion-eligible windows]
		The restriction to torsion-eligible windows is essential for the present formulation, since it allows all Whitehead groups to be identified with a single group $\mathrm{Wh}(G)$. 
		A natural direction for future work is to retain the variation of the groups 
		\[
		\mathrm{Wh}(\pi_1(K_t))
		\]
		along a filtration, together with the maps induced by the filtration inclusions. 
		Such a construction would remove the torsion-eligibility hypothesis and would provide a more intrinsic persistent version of the classical Whitehead torsion composition law.
	\end{remark}

	\smallskip
	\noindent
	We now illustrate the non-triviality of persistent Whitehead torsion with an explicit example using lens spaces.

	\begin{example}[Lens space filtration]\label{eg:lensspace}
		For integers $p\geq 2$ and $q$ relatively
		prime to $p$, let $L(p;1,q)$ be the classical $3$-dimensional lens spaces. Its fundamental group is given by
		\begin{equation}\label{eq:lensfund}
			\pi_1\left(L(p;1,q)\right)
			\cong
			\mathbb Z_p.
		\end{equation}
		and the homology groups are given by
		\begin{equation}\label{eq:lenshom}
			H_i\bigl(L(p;1,q)\bigr)=
			\begin{cases}
				\mathbb Z, & i=0,3,\\
				\mathbb Z_{p}, & i=1,\\
				0, & \text{otherwise}.
			\end{cases}
		\end{equation} 
		
		\smallskip
		\noindent 
		It is well known that the two lens spaces $L(p;1,q)$ and $L(p;1,q')$ are homotopy equivalent if and only if $qq' \equiv \pm n^2 \pmod{p}$ for some integer $n$, and simple-homotopy equivalent if and only if $q \equiv \pm q' \pmod{p}$ (see \cite{JHC}). In particular, $L(7;1,1)$ and $L(7;1,2)$ are homotopy equivalent but \emph{not} simple-homotopy equivalent.
		
		\smallskip 
		\noindent 	
		We consider $X=L(7;1,1)$ and $Y=L(7;1,2)$. Since $X\simeq Y$, let $f:X\to Y$ be a homotopy equivalence with the homotopy inverse
		$g:Y\to X$.
		Let $I=[0,1]$. Denote by
		$$
		M_f=\bigl((X\times I)\sqcup Y\bigr)\big/\bigl((x,1)\sim f(x)\bigr),
		\qquad
		M_g=\bigl((Y\times I)\sqcup X\bigr)\big/\bigl((y,1)\sim g(y)\bigr).
		$$
		the corresponding mapping cylinders. Let $i_X:X\hookrightarrow M_f$ and $i_Y:Y\hookrightarrow M_g$ denote
		the canonical inclusions. It is well known that $M_f$ deformation retracts onto $Y$ and
		$M_g$ deformation retracts onto $X$ (see \cite{AT}). Since $f$ and $g$
		are homotopy equivalences, it follows that $i_X$ and $i_Y$ are also
		homotopy equivalences. Consequently (using equation \eqref{eq:lensfund}), we obtain
		$$
		\pi_1(M_f)\cong\pi_1(Y)\cong\pi_1(X)
		\cong\pi_1(M_g)\cong\mathbb Z_7.
		$$
		
		\smallskip
		\noindent
		Since $M_f$ and $M_g$ contain a common copy of $Y$, we glue them along
		$Y$ and define
		$$
		M_f\cup_Y M_g.
		$$
		Let
		$$
		K_0=X,\qquad
		K_1=M_f,\qquad
		K_2=M_f\cup_Y M_g,
		$$
		which yields a filtration
		$
		K_0\subset K_1\subset K_2.
		$ 
		Since the inclusion
		$$
		\iota_{0,1}:K_0\hookrightarrow K_1
		$$
		is a homotopy equivalence, we consider the subset
		$
		J=\{0,1\}
		$ 
		a torsion-eligible window for the filtration. Moreover,
		$$
		\tau(\iota_{0,1})
		=
		\tau(i_X)
		=
		\tau(f)
		\neq0
		\qquad\text{in}\qquad
		\mathrm{Wh}(\mathbb Z
		_7)
		\cong\mathbb Z^2.
		$$
		Hence,
		$
		\matholdcal T(K_\bullet)(0\le1)
		\neq0.$
		On the other hand, the inclusion
		$$
		K_0=X\hookrightarrow K_1=M_f
		$$
		is a homotopy equivalence and therefore induces isomorphisms on homology in every degree. Nevertheless, its Whitehead torsion is non-trivial.
		Thus, this example illustrates that persistent Whitehead torsion captures simple-homotopy information not determined by persistent homology.
	\end{example}
	
	\begin{define}
		Let $K_0\subseteq K_1\subseteq\cdots\subseteq K_n$ be a torsion-eligible filtration. Its \emph{elementary torsion profile} is the sequence
		\[\mathcal{E}(K_\bullet)=(e_0,e_1,\ldots,e_{n-1}),\]
		where $e_i:=\tau_{i,i+1}\in\operatorname{Wh}(G)$.
	\end{define}
	
	\begin{remark}\label{rem:elemtor}
		The elementary torsion profile is a complete encoding of the persistent Whitehead torsion. Indeed, every pairwise torsion is recovered by summing the elementary torsions along the corresponding interval:
		\[
		\tau_{i,j}=\sum_{k=i}^{j-1}e_k.
		\]
		Conversely, the elementary torsion profile is obtained from the persistent torsion assignment by restricting to consecutive pairs.
		
		\smallskip
		\noindent
		Thus, the profile gives a linear-size description of the persistent torsion data: rather than recording all pairwise torsions, it is enough to record the torsion of the consecutive transitions. 
		In particular, $\matholdcal{T}(K_\bullet)$ vanishes precisely when the elementary torsion profile is identically zero.
		
		\smallskip
		\noindent
		It is worth noting that this transition-level information can contain information that is invisible from a single endpoint torsion. For instance, consecutive elementary torsions may cancel in $\operatorname{Wh}(G)$, so that $\tau_{i,i+2}=0$ even though both $e_i$ and $e_{i+1}$ are nonzero. The elementary profile therefore records the history of the simple-homotopy changes along the filtration rather than only their accumulated effect over a longer interval.
	\end{remark}
	
	\begin{remark}
		The Morse complexity profile and the elementary torsion profile provide two complementary descriptions of the evolution of a filtration. The Morse complexity profile associates a quantity to each filtration level, recording the complexity of the corresponding space and its discrete Morse reductions. In contrast, the elementary torsion profile is attached to the transitions between consecutive levels and records the simple-homotopy information carried by the inclusion maps.
		
		\smallskip \noindent Thus, if the Morse complexity profile is written schematically as $M(K_t)$, it is a level-based invariant: it asks how the individual spaces $K_t$ change as $t$ varies. The elementary torsion profile, on the other hand, consists of the quantities $e_i=\tau_{i,i+1}$, and is therefore transition-based: it asks how the simple-homotopy type changes from one level to the next.
		
		\smallskip \noindent
		The distinction is particularly useful because the elementary torsion profile is not merely a summary of the pairwise torsion data. By the proposition above, it completely determines that data: every torsion $\tau_{i,j}$ is obtained by summing the elementary torsions along the interval from $i$ to $j$. At the same time, the individual elementary torsions retain information that can be lost through cancellation when only longer intervals are considered.
		
		\smallskip
		\noindent
		The two profiles therefore capture different aspects of the filtration. Morse complexity profile measures the combinatorial complexity of the spaces appearing in the filtration, whereas Whitehead torsion measures an obstruction to the homotopy equivalences between consecutive spaces being simple-homotopy equivalences. They should consequently be viewed as complementary persistent summaries rather than as competing invariants: the Morse complexity profile describes the evolution of the spaces, while the elementary torsion profile describes the simple-homotopy behavior of the transitions between them.
	\end{remark}
	
	\smallskip
	\noindent The following results show that the persistent Whitehead torsion is invariant under the level-wise homotopy equivalence and simple-homotopy equivalences of the filtrations.

	\smallskip
	\noindent
	Let $K_\bullet=\{K_t\}_{t\in I}$ and
	$L_\bullet=\{L_t\}_{t\in I}$ be filtrations of finite connected
	simplicial complexes, and let $J\subseteq I$ be a torsion-eligible
	window for both. Denote their inclusion maps by
	$\iota^K_{s,t}$ and $\iota^L_{s,t}$, respectively.
	Recall that the torsion classes for $K_\bullet$ are defined in
	\eqref{eq:tor} using compatible identifications
	$\alpha_t:\pi_1(K_t)\xrightarrow{\cong}G$. We choose
	\[
	\beta_t:\pi_1(L_t)\xrightarrow{\cong}G
	\]
	so that
	\[
	\beta_t\circ(f_t)_*=\alpha_t.
	\]
	Thus, under the identifications $\alpha_t$ and $\beta_t$,
	$
	(f_t)_*:\pi_1(K_t)\longrightarrow\pi_1(L_t)
	$
	is identified with the identity automorphism of $G$.
	
	\smallskip 
	\noindent For $s\leq t$ in $J$, we set
	\[
	\tau^L_{s,t}
	:=
	\beta_{t*}^{\mathrm{Wh}}
	\bigl(\tau(\iota^L_{s,t})\bigr)
	\in\mathrm{Wh}(G).
	\]
	\begin{theorem}\label{thm:pwt-he}
		Let $f_\bullet:K_\bullet=\{K_t\}_{t\in J}\to
		L_\bullet=\{L_t\}_{t\in J}$
		be a filtration map between torsion-eligible filtrations of finite
		connected simplicial complexes 
		Suppose that, for each $t\in J$,
		$f_t:K_t\to L_t$
		is a homotopy equivalence. With the compatible
		identifications and torsion classes fixed above, there is a natural
		isomorphism
		\[
		\eta:\matholdcal T(K_\bullet)
		\Longrightarrow
		\matholdcal T(L_\bullet).
		\]

	\end{theorem}
	
	\begin{proof}
		For each $t\in J$, define
		\[
		\eta_t:=\beta_{t_\ast}^{\mathrm{Wh}}(\tau(f_t))\in
		\operatorname{Hom}_{\widehat{\mathrm{Wh}(G)}}(*,*)
		=\mathrm{Wh}(G).
		\]
		\noindent
		Since $\widehat{\mathrm{Wh}(G)}$ is a one-object groupoid,
		each $\eta_t$ is an isomorphism. We now show that the family
		$\{\eta_t\}_{t\in J}$ defines a natural transformation $
		\eta:\matholdcal T(K_\bullet)\Rightarrow\matholdcal T(L_\bullet),$
		i.e., the following diagram
		\[
		\begin{tikzcd}[column sep=5em]
			\matholdcal T(K_\bullet)(s)
			\arrow[r,"\matholdcal T(K_\bullet)(s\le t)"]
			\arrow[d,"\eta_s"']
			&
			\matholdcal T(K_\bullet)(t)
			\arrow[d,"\eta_t"]\\
			\matholdcal T(L_\bullet)(s)
			\arrow[r,"\matholdcal T(L_\bullet)(s\le t)"']
			&
			\matholdcal T(L_\bullet)(t)
		\end{tikzcd}
		\]
		commutes for every $s\le t$ in $J$.
		
		\smallskip\noindent
		Let $s\le t$ in $J$.
		Since $f_{\bullet}:K_\bullet=\{K_t\}_{t\in I} \longrightarrow \{L_t\}_{t\in I}=L_\bullet$ is a filtration map, we have 
		$$
		f_t \circ \iota^K_{s,t} =\iota^L_{s,t} \circ f_s
		$$
		for all $s \le t$. 
		Taking Whitehead torsion on both sides gives \begin{equation}\label{eq5.1}
			\tau(f_t \circ \iota^K_{s,t}) = \tau(\iota^L_{s,t} \circ f_s).
		\end{equation}
		\noindent Further, by the composition formula of the Whitehead torsion (see \cite[Statement 22.4]{Coh}), we get 
		\begin{equation*}
			\tau(f_t \circ \iota^K_{s,t}) 
			= \tau(f_t) + (f_t)_{\ast}^{\mathrm{Wh}}(\tau(\iota^K_{s,t})) \quad \textup{and}\quad  \tau(\iota^L_{s,t} \circ f_s) 
			= \tau(\iota^L_{s,t}) + (\iota^L_{s,t})_{\ast}^{\mathrm{Wh}}(\tau(f_s)),
		\end{equation*}
		where $(f_t)_{\ast}^{\mathrm{Wh}}: \mathrm{Wh}(\pi_1(K_t)) \longrightarrow \mathrm{Wh}(\pi_1(L_t))$ and $(\iota^L_{s,t})_{\ast}^{\mathrm{Wh}}: \mathrm{Wh}(\pi_1(L_s)) \longrightarrow \mathrm{Wh}(\pi_1(L_t))$ are the isomorphisms induced by $f_t$ and $\iota^L_{s,t},$ respectively.
		Under the chosen identifications with the fixed group $G$, we obtain
		\begin{align}\label{eq:toradd1}
			\beta_{t_\ast}^{\mathrm{Wh}}(\tau(f_t)) + \beta_{t_\ast}^{\mathrm{Wh}}(f_t)_{\ast}^{\mathrm{Wh}}(\tau(\iota^K_{s,t})))
			&= \beta_{t_\ast}^{\mathrm{Wh}}(\tau(\iota^L_{s,t})) + \beta_{t_\ast}^{\mathrm{Wh}}((\iota^L_{s,t})_{\ast}^{\mathrm{Wh}}(\tau(f_s)))\\\notag
			\eta_t+\alpha_{t_\ast}^{\mathrm{Wh}}(\tau(\iota^K_{s,t}))&=\beta_{s_\ast}^{\mathrm{Wh}}(\tau(\iota^L_{s,t}))+\eta_s\tag{since $\beta_t\circ (\iota_{s,t}^L)_{*}=\beta_s$ }\\
			\eta_t+\tau_{s,t}^K&=\tau_{s,t}^L+\eta_s\tag{by~\eqref{eq:tor}}
		\end{align}
		
		\noindent Since composition in
		$\widehat{\mathrm{Wh}(G)}$
		is given by addition in
		$\mathrm{Wh}(G)$,
		this is precisely the naturality condition
		\[
		\eta_t\circ
		\matholdcal T(K_\bullet)(s\le t)
		=
		\matholdcal T(L_\bullet)(s\le t)
		\circ
		\eta_s.
		\]
		Hence, the family
		$\{\eta_t\}_{t\in J}$
		defines a natural isomorphism
		\[ \matholdcal T(K_\bullet) \Longrightarrow \matholdcal T(L_\bullet).\qedhere
		\]
	\end{proof}
	
	\begin{cor}\label{cor:pwt-she}
		Under the hypotheses of Theorem~\ref{thm:pwt-he}, suppose in addition that for
		each $t\in J$,
		$f_t:K_t\to L_t$
		is a simple-homotopy equivalence.
		Then,
		\[
		\matholdcal T(K_\bullet)
		=
		\matholdcal T(L_\bullet).
		\]
	\end{cor}
	\begin{proof}
		Since each $f_t$ is a simple-homotopy equivalence,
		$\tau(f_t)=0$ for every $t\in J$. Hence, the components
		\[
		\eta_t= \beta_{t_\ast}^{\mathrm{Wh}}(\tau(f_t))
		\]
		of the natural isomorphism constructed in
		Theorem~\ref{thm:pwt-he} are the identity morphisms of
		$\widehat{\mathrm{Wh}(G)}$. Therefore, the natural isomorphism is the
		identity natural transformation, and consequently
		\[
		\matholdcal T(K_\bullet)=\matholdcal T(L_\bullet).
		\]
		
		\smallskip\noindent
		Equivalently, one may observe directly from
		equation~\eqref{eq:toradd1} that, since
		$\tau(f_s)=\tau(f_t)=0$, the equation \eqref{eq:toradd1} reduces to
		\begin{equation}
			\tau^K_{s,t} = \tau^L_{s,t},
		\end{equation}
		for every morphism $s\le t$ in $P(J)$.
		Moreover, for every $t\in J$, both functors send $t$ to the unique object $\{\ast\}$
		of $\widehat{\mathrm{Wh}(G)}$. Therefore, $
		\matholdcal T(K_\bullet)
		=
		\matholdcal T(L_\bullet).$
	\end{proof}
	
	\begin{lemma}\label{lem:torelig}
		Let
		$K_\bullet=\{K_t\}_{t\in I}$
		and
		$L_\bullet=\{L_t\}_{t\in I}$
		be $\epsilon$-interleaved filtrations of finite connected simplicial complexes.
		Let $J\subseteq I$ be a torsion-eligible window for  $K_\bullet$ and let the interleaving map $$
		\phi_t:K_t\longrightarrow L_{t+\epsilon}
		$$
		is a homotopy equivalence for all $t\in J$. Then, $J+\epsilon$ is a torsion-eligible
		window for $L_\bullet$.
	\end{lemma}
	
	\begin{proof}
		Let $s\le t$ in $J$. Since $J$ is a torsion-eligible window for $K_\bullet$, the inclusion $\iota^K_{s,t}:K_s\to K_t
		$ is a homotopy equivalence. Further, as each $\phi_t$ is a homotopy equivalence, the composition $
		\phi_t\circ\iota^K_{s,t}$ is a homotopy equivalence. Since the filtrations 
		$K_\bullet$
		and
		$L_\bullet$
		are $\epsilon$-interleaved, from Definition \ref{def:interleaving}, we have
		\[
		\phi_t\circ\iota^K_{s,t}
		=
		\iota^L_{s+\epsilon,t+\epsilon}\circ\phi_s.
		\]
		It follows that
		$\iota^L_{s+\epsilon,t+\epsilon}\circ\phi_s$
		is also a homotopy equivalence.
		Finally, as 
		$
		\phi_s
		$ and $\iota^L_{s+\epsilon,t+\epsilon}\circ\phi_s$
		are homotopy equivalences, the $2$-out-of-$3$ property for homotopy equivalences (see, for instance, \cite[Chapter~0]{AT}) implies that
		$
		\iota^L_{s+\epsilon,t+\epsilon}
		$
		is a homotopy equivalence for all $s\leq t$ in $J$. Hence, $J+\epsilon$ is a torsion-eligible window for $L_\bullet$.
	\end{proof}
	
	\noindent Let $J\subseteq \mathbb{R}_{\geq 0}$ and 
	$J+\epsilon=\{t+\epsilon:t\in J\}$. Following \cite[\S 2]{BDSS},
	we define the \emph{translation functor}
	\[
	\Sigma_\epsilon:P(J)\longrightarrow P(J+\epsilon)
	\]
	by
	$\Sigma_\epsilon(t)=t+\epsilon$ on objects and
	$\Sigma_\epsilon(s\leq t)=(s+\epsilon\leq t+\epsilon)$ on morphisms.

	\begin{theorem}[Interleaving invariance of persistent Whitehead torsion]\label{thm:pwt-interleaving-he}
		Let $K_\bullet=\{K_t\}_{t\in I}$ and
		$L_\bullet=\{L_t\}_{t\in I}$ be $\epsilon$-interleaved filtrations of
		finite connected simplicial complexes. Let $J\subseteq I$ be a
		torsion-eligible window for $K_\bullet$, and suppose that the
		interleaving maps
		\[
		\phi_t:K_t\longrightarrow L_{t+\epsilon}
		\]
		are homotopy equivalences for all $t\in J$. Then $J+\epsilon$ is a torsion-eligible window for $L_\bullet$. Choose compatible identifications
		\[
		\alpha_t:\pi_1(K_t)\xrightarrow{\cong}G,
		\qquad
		\beta_{t+\epsilon}:\pi_1(L_{t+\epsilon})\xrightarrow{\cong}G,
		\qquad t\in J,
		\]
		such that, for every $s\leq t$ in $J$,
		\begin{align}\label{eq:interleaving-identifications}
			\alpha_t\circ(\iota^K_{s,t})_*=\alpha_s,
			\qquad
			\beta_{t+\epsilon}\circ
			(\iota^L_{s+\epsilon,t+\epsilon})_*
			=\beta_{s+\epsilon},\qquad
			\beta_{t+\epsilon}\circ(\phi_t)_*=\alpha_t.
		\end{align}
		Then, there is a natural isomorphism
		\[
		\eta:
		\matholdcal T(K_\bullet|_J)
		\Longrightarrow
		\matholdcal T(L_\bullet|_{J+\epsilon})\circ\Sigma_\epsilon.
		\]
		
	\end{theorem}

	\begin{proof}
		Since $J$ is torsion-eligible window for $K_{\bullet}$, it follows from Lemma \ref{lem:torelig} that $J+\epsilon$ is a torsion-eligible window for $L_{\bullet}.$ 
		Since
		$K_\bullet$
		and
		$L_\bullet$
		are $\epsilon$-interleaved (see  Definition \ref{def:interleaving}), the family of interleaving maps
		$\{\phi_t:K_t\to L_{t+\epsilon}\}_{t\in J}$ satisfies 
		\[
		\phi_t\circ\iota^K_{s,t}
		=
		\iota^L_{s+\epsilon,t+\epsilon}\circ\phi_s
		\]
		for all $s\le t$ in $J$. Then, by the 
		the composition formula for the Whitehead torsion
		\cite[Statement~22.4]{Coh}, we get
		\begin{align}
			\tau(\phi_t\circ\iota^K_{s,t})
			&=
			\tau(\iota^L_{s+\epsilon,t+\epsilon}\circ\phi_s)\\\notag
			\tau(\phi_t)
			+
			(\phi_t)_*^{\mathrm{Wh}}(\tau(\iota^K_{s,t}))
			&=
			\tau(\iota^L_{s+\epsilon,t+\epsilon})
			+
			(\iota^L_{s+\epsilon,t+\epsilon})_*^{\mathrm{Wh}}(\tau(\phi_s)).
		\end{align}
		Define $\eta_t:=\beta_{(t+\epsilon)_*}^{\mathrm{Wh}}(\tau(\phi_t))\in
		\operatorname{Hom}_{\widehat{\mathrm{Wh}(G)}}(\ast, \ast)
		=\mathrm{Wh}(G)$. Then, we have
		\begin{align*}
			\beta_{(t+\epsilon)_*}^{\mathrm{Wh}}(\tau(\phi_t))
			+
			\beta_{(t+\epsilon)_*}^{\mathrm{Wh}}((\phi_t)_*^{\mathrm{Wh}}(\tau(\iota^K_{s,t})))&=\beta_{(t+\epsilon)_*}^{\mathrm{Wh}}(\tau(\iota^L_{s+\epsilon,t+\epsilon}))
			+
			\beta_{(t+\epsilon)_*}^{\mathrm{Wh}}((\iota^L_{s+\epsilon,t+\epsilon})_*^{\mathrm{Wh}}(\tau(\phi_s)))\\\notag
			\eta_t+\alpha_{t*}^{\mathrm{Wh}}(\tau(\iota^k_{s,t}))&= \tau^L_{s+\epsilon,t+\epsilon}+\beta_{(s+\epsilon)_*}^{\mathrm{Wh}}(\tau(\phi_s))\tag{by~\eqref{eq:interleaving-identifications}}\\
			\eta_t+\tau^K_{s,t}&=\tau^L_{s+\epsilon,t+\epsilon}+\eta_s
		\end{align*}
		Therefore, 
		\[
		\eta_t\circ
		\matholdcal T(K_\bullet)(s\le t)
		=
		\matholdcal T(L_\bullet)(s+\epsilon\le t+\epsilon)
		\circ
		\eta_s.
		\]
		
		This completes the proof.
	\end{proof}
	\noindent The analogous statement holds in the opposite direction. Namely, if
	$J$ is a torsion-eligible window for $L_\bullet$ and $
	\psi_t:L_t\longrightarrow K_{t+\epsilon}$ is a homotopy equivalence for every $t\in J$, then
	$J+\epsilon$ is a torsion-eligible window for $K_\bullet$.
	Choosing compatible identifications
	\[
	\beta_t:\pi_1(L_t)\xrightarrow{\cong}G,
	\qquad
	\alpha_{t+\epsilon}:\pi_1(K_{t+\epsilon})\xrightarrow{\cong}G,
	\]
	satisfying
	\[
	\alpha_{t+\epsilon}\circ(\psi_t)_*=\beta_t,
	\]
	and the corresponding compatibility conditions along the two
	filtrations, the same argument gives a natural isomorphism
	\[
	\zeta:
	\matholdcal T(L_\bullet|_J)
	\Longrightarrow
	\matholdcal T(K_\bullet|_{J+\epsilon})\circ\Sigma_\epsilon.
	\]
	\begin{cor}\label{cor:pwt-interleaving-she}
		Under the hypotheses of Theorem~\ref{thm:pwt-interleaving-he}, suppose in addition that for all $t\in J$, the interleaving maps
		$$
		\phi_t:K_t\longrightarrow L_{t+\epsilon}
		$$ is simple-homotopy equivalence. Then,
		$$
		\matholdcal{T}(K_\bullet|_J)
		=
		\matholdcal T(L_\bullet|_J)\circ\Sigma_\epsilon
		$$
	\end{cor}
	\begin{proof}
		Since each $\phi_t$ is a simple-homotopy equivalence, $
		\tau(\phi_t)=0.$ 
		Thus, each component $\eta_t$ of the natural isomorphism constructed
		in Theorem~\ref{thm:pwt-interleaving-he} is the identity morphism of
		$\widehat{\mathrm{Wh}(G)}$. Hence,
		\[
		\matholdcal T(K_\bullet|_J) = \matholdcal T(L_\bullet|_{J+\epsilon})\circ\Sigma_\epsilon.\qedhere \]
	\end{proof}
	\noindent 
	The analogous conclusion holds for the other interleaving map. If for all $t\in J,$ each $
	\psi_t:L_t\longrightarrow K_{t+\epsilon},$ 
	is a simple-homotopy equivalence, then we have
	\[
	\matholdcal T(L_\bullet|_J)
	=
	\matholdcal T(K_\bullet|_{J+\epsilon})\circ\Sigma_\epsilon.
	\]
	
	\begin{remark}[A hierarchy of persistent information]
		The construction above suggests a hierarchy of information carried by a filtration. Persistent homology records the evolution of homological type, while on a torsion-eligible window the inclusion maps are already homotopy equivalences. 
		The persistent Whitehead torsion then measures the remaining obstruction to these homotopy equivalences being simple-homotopy equivalences. 
		In this sense, the invariant lies strictly beyond persistent homology: it is sensitive to information that survives after homological and homotopical changes have disappeared.
	\end{remark}


\begin{thebibliography}{99}
		
		\bibitem{HA} H. Adams et al., Persistence images: a stable vector representation of persistent homology, \textit{J. Mach. Learn. Res.} {\bf 18} (2017), Paper No. 8, 35 pp.
		\bibitem{AB} K. A. Adiprasito and B. Benedetti. Collapsibility of ${\rm CAT}(0)$ spaces, \textit{Geom. Dedicata} {\bf 206} (2020), 181--199.
		\bibitem{AL17}
		M.~Allili, T.~Kaczynski, and C.~Landi,
		Reducing complexes in multidimensional persistent homology theory,
		\emph{J. Symbolic Comput.} \textbf{78} (2017), 61--75.
		\bibitem{Rip}
		U.~Bauer,
		\newblock Ripser: efficient computation of Vietoris--Rips persistence
		barcodes,
		\newblock \emph{J. Appl. Comput. Topol.} \textbf{5} (2021), 391--423.
		\bibitem{BS}
		P.~Bubenik and J.~A.~Scott,
		\emph{Categorification of persistent homology},
		Discrete Comput. Geom. \textbf{51} (2014), no.~3, 600--627.
		\bibitem{BDSS} P.~Bubenik, V.~de Silva, and J.~Scott, \emph{Metrics for generalized persistence modules}, Found. Comput. Math. \textbf{15} (2015), no.~6, 1501--1531.
		\bibitem{CH} M. K. Chari, On discrete Morse functions and combinatorial decompositions, \textit{Discrete Mathematics},
		{\bf217}(1-3) (2000), 101–113.
		\bibitem{Coh} M. M. Cohen, A course in simple-homotopy theory, Graduate Texts in Mathematics, {\bf10},
		Springer-Verlag, New York-Berlin, (1973).
		
		\bibitem{ELZ} H. Edelsbrunner, D. Letscher and A.~J. Zomorodian, Topological persistence and simplification, \textit{Discrete Comput. Geom.} {\bf 28} (2002), no.~4, 511--533
		\bibitem{EH} H. Edelsbrunner and J.~L. Harer,  Computational topology, \textit{Amer. Math. Soc.}, Providence, RI, 2010.
		\bibitem{EdelsDepth}
		H.~Edelsbrunner, M.~Lipi\'nski, M.~Mrozek, and
		M.~Soriano-Trigueros,
		\newblock The poset of cancellations in a filtered complex,
		\newblock arXiv:2311.14364.
		\bibitem{EdelTranspositions}
		H.~Edelsbrunner, M.~Lipi\'nski, M.~Mrozek,
		M.~Soriano-Trigueros, and F.~Zimin,
		\newblock The depth poset under transpositions in the filter,
		\newblock in \emph{42nd International Symposium on Computational
			Geometry (SoCG 2026)}, Leibniz International Proceedings in
		Informatics (LIPIcs), vol.~367, 2026, pp.~41:1--41:18.
		\bibitem{XF} X. L.  Fernández, Strong discrete Morse theory, arXiv: 2504.15729, (2025), 
		\url{https://arxiv.org/abs/2504.15729}.
		\bibitem{Hersh}
		P.~Hersh,
		\emph{On optimizing discrete Morse functions},
		Adv. in Appl. Math. \textbf{35} (2005), no.~3, 294--322.
		
		\bibitem{KSRS}
		H.~Kannan, E.~Saucan, I.~Roy, and A.~Samal,
		Persistent homology of unweighted complex networks via discrete Morse theory,
		\emph{Sci. Rep.} \textbf{9} (2019), 13817.
		
		\bibitem{FIF}
		U.~Fugacci, F.~Iuricich, and L.~De Floriani,
		Computing discrete Morse complexes from simplicial complexes,
		\emph{Graph. Models} \textbf{103} (2019), 101023.
		\bibitem{RB98} R. Forman,  Morse theory for cell complexes. \textit{Adv. Math.} {\bf134} (1) (1998), 90--145.
		\bibitem{RB02}
		R. Forman, A user's guide to discrete Morse theory, \textit{S\'em. Lothar. Combin.} {\bf 48} (2002), Art.\ B48c, 35 pp.
		\bibitem{AT}
		A. Hatcher,
		Algebraic Topology,
		\textit{Cambridge University Press}, (2002).
		\bibitem{JP}
		M.~Joswig and M.~E.~Pfetsch,
		\emph{Computing optimal Morse matchings},
		SIAM J. Discrete Math. \textbf{20} (2006), no.~1, 11--25.
		\bibitem{LLT}
		T.~Lewiner, H.~Lopes, and G.~Tavares,
		\emph{Towards optimality in discrete Morse theory},
		Experiment. Math. \textbf{12} (2003), no.~3, 271--285.
		\bibitem{LagoasEtAl}
		J.~Lagoas, T.~Lewiner, T.~Novello, and J.~Paix\~ao,
		\emph{Greedy Morse Matching and Discrete Smoothness}, arXiv:1801.10118 [math.GT], 2018.
		\bibitem{LewinerThesis}
		T.~Lewiner,
		\emph{Geometric Discrete Morse Complexes},
		Ph.D. thesis, Pontifical Catholic University of Rio de Janeiro, 2005.
		\bibitem{MS}
		C.~Maria and H.~Schreiber,
		Discrete Morse theory for computing zigzag persistence,
		\emph{Discrete Comput. Geom.} \textbf{71} (2024), no.~2, 708--737.
		\bibitem{JM} J. Milnor, Whitehead torsion, \textit{Bulletin of the Amer. Math. Soc.} \textbf{72}(3), (1966), 358--426.
		\bibitem{MN} K. Mischaikow and V. Nanda, Morse theory for filtrations and efficient computation of persistent homology, \textit{Discrete Comput. Geom.} {\bf 50} (2013), no.~2, 330--353.
		\bibitem{VN} V Nanda. Computational algebraic topology lecture notes. url: https://people.maths.ox.ac.uk/nanda/cat/TDANotes. pdf, 2021.
		
		\bibitem{OU}S. Y.
		Oudot, Persistence theory: from quiver representations to data analysis,
		{\bf209}, Mathematical Surveys and Monographs. \textit{Amer. Math.  Soc.}, Providence, RI, 2015.
		\bibitem{PV}
		G.~Philip and E.~R.~V{\aa}gset,
		\emph{ETH-tight complexity of optimal Morse matching on bounded-treewidth complexes},
		in \emph{42nd International Symposium on Computational Geometry (SoCG 2026)},
		Leibniz International Proceedings in Informatics (LIPIcs), 2026,
		Article~85, 85:1--85:19.
		\bibitem{ALQ}
		A.~L.~Qiu,
		Simple Homotopy Equivalence and Whitehead Groups,
		unpublished manuscript, available at \url{https://www.math.toronto.edu/qiu/writings/SimpleHomotopy.pdf}.
		\bibitem{SN19}N. A. Scoville,  Discrete Morse Theory, \textit{Amer.  Math. Soc.} {\bf 90},(2019).
		\bibitem{JHC} J. H. C. Whitehead, Simple homotopy types, \textit{American Journal of Mathematics} \textbf{72}(1), (1950), 1--57.
		\bibitem{ZC05}
		A. Zomorodian and G. Carlsson,
		Computing persistent homology,
		\textit{Discrete Comput. Geom.} \textbf{33} 2 (2005), 249--274.
		
	\end{thebibliography}
\end{document}